\documentclass{article}
\usepackage[utf8]{inputenc}
\usepackage{Stylefiles/pers}
\usepackage{Stylefiles/alias}
\usepackage{Stylefiles/theorems}
\usepackage{Stylefiles/arxiv}

\newcommand{\qedMP}{}

\title{Improved convergence analysis of Lasserre's measure--based upper bounds for polynomial minimization on compact sets}
\runningtitle{Improved convergence analysis of Lasserre's measure--based upper bounds for polynomial minimization}
\author{
	Lucas Slot \thanks{Centrum Wiskunde \& Informatica (CWI), Amsterdam, \url{lucas.slot@cwi.nl}}
	\And
	Monique Laurent \thanks{Centrum Wiskunde \& Informatica (CWI), Amsterdam and Tilburg University, \url{monique.laurent@cwi.nl}}
}
\date{\today}

\begin{document}

\maketitle
\begin{abstract}
	We consider the problem of computing the minimum value $f_{\min,K}$ of a polynomial $f$ over a compact set $K\subseteq \R^n$, which can be reformulated as finding a probability measure $\nu$ on $\mainset$ minimizing $\int_\mainset f d\nu$. Lasserre showed that it suffices to consider such measures of the form $\nu = q\mu$, where $q$ is a sum-of-squares polynomial and $\mu$ is a given  Borel measure supported on $\mainset$. By bounding the degree of $q$ by $2r$ one gets a converging hierarchy of upper bounds $f^{(r)}$\, for $f_{\min,K}$. 
When $K$ is the hypercube $[-1, 1]^n$, equipped with the Chebyshev measure, the parameters $\Lasupbound{r}$ are known to converge to $\mainfuncmin$ at a rate in $\bigO(1/r^2)$. We extend this error estimate to a wider class of convex bodies, while also allowing for a broader class of reference measures, including the Lebesgue measure. Our analysis applies to 	simplices,  balls and convex bodies that locally look like a ball.
In addition,  we show an error estimate in $\bigO(\log r / r)$ when $\mainset$ satisfies a minor geometrical condition, and in  $\bigO(\log^2 r / r^2)$ when $\mainset$ is a convex body, equipped with the Lebesgue measure.
This improves upon the currently best known error estimates in $\bigO(1 / \sqrt{r})$ and $\bigO(1/r)$ for these two respective cases.

\end{abstract}
	
\section{Introduction}
    \label{SEC_Introduction}
    \subsection{Lasserre's measure-based hierarchy}

Let $\mainset \subseteq \R^n$ be a compact set and let $\mainfunc \in \R[x]$ be a polynomial. We consider the minimization problem
\begin{equation}
    \label{EQ_mainfuncmin}
    \mainfuncmin := \min_{x \in \mainset} \mainfunc(x).
\end{equation}

Computing $\mainfuncmin$ is a hard problem in general, and some well-known problems from combinatorial optimization are among its special cases. For example, it is shown in \cite{MotzkinStraus1965, deKlerkPasechnick2002} that the stability number $\alpha(G)$ of a graph $G = ([n], E)$ is given by
\begin{equation}
    \frac{1}{\alpha(G)} = \min_{x \in K} \sum_{i \in V} x^2_i + 2\sum_{\{i,j\} \in E} x_ix_j,
\end{equation}
where we take  $K= \{x \in \R^n : x \geq 0, \ \sum_{i=1}^n x_i = 1\}$ to be the standard simplex in $\R^n$.

Problem \eqref{EQ_mainfuncmin} may be reformulated as the problem of  finding a probability measure $\nu$ on $\mainset$ for which the integral $\int_\mainset \mainfunc d\nu$ is minimized. Indeed, for any such $\nu$ we have $\int_\mainset f d\nu \geq \mainfuncmin \int_\mainset d \nu = \mainfuncmin$. On the other hand, if $a \in \mainset$ is a global minimizer of $\mainfunc$ in $K$, then we have $\int_\mainset \mainfunc d\delta_a = \mainfunc(a) =  \mainfuncmin$, where $\delta_a$ is the Dirac measure centered at $a$. 

Lasserre \cite{Lasserre2010} showed that it suffices to consider measures of the form $\nu = \Laspoly \mu$, where $\Laspoly \in \Sigma$ is a sum-of-squares polynomial and $\mu$ is a (fixed) reference Borel measure supported by $\mainset$. That is, we may reformulate \eqref{EQ_mainfuncmin} as
\begin{equation}
    \label{EQ_fmin_measures}
    \mainfuncmin = \inf_{\Laspoly \in \Sigma} \int_\mainset \mainfunc(x)\Laspoly(x)d\mu(x) \quad\text{s.t.} \int_\mainset \Laspoly(x)d\mu(x) = 1.
\end{equation}
For each $r \in \N$ we may then obtain an upper bound $\Lasupboundl{r}{\mainset, \mu}$ for $\mainfuncmin$ by limiting our choice of $\Laspoly$ in \eqref{EQ_fmin_measures} to polynomials of degree at most $2r$:
\begin{equation}
    \label{EQ_fmin_measures_bounded}
    \Lasupboundl{r}{\mainset, \mu} := \inf_{\Laspoly \in \Sigma_r} \int_\mainset \mainfunc(x)\Laspoly(x)d\mu(x) \quad\text{s.t.} \int_\mainset \Laspoly(x)d\mu(x) = 1.
\end{equation}
Here, $\Sigma_r$ denotes the set of all sum-of-squares polynomials of degree at most $2r$.
We shall also write $\Lasupbound{r} = \Lasupboundl{r}{\mainset, \mu}$ for simplicity.
As detecting sum-of-squares polynomials is possible using semidefinite programming, the program \eqref{EQ_fmin_measures_bounded} can be modeled as an SDP \cite{Lasserre2010}. Moreover, the special structure of this SDP allows a reformulation to an eigenvalue minimization problem \cite{Lasserre2010}, as will be briefly described below.

By definition, we have $\mainfuncmin \leq \Lasupbound{r+1} \leq \Lasupbound{r}$ for all $r \in \N$ and 
\begin{equation}
    \label{EQ_converging_sequence}
    \lim_{r \rightarrow \infty} \Lasupbound{r} = \mainfuncmin.
\end{equation}
In this paper we are interested in upper bounding the convergence rate of the sequence $(\Lasupbound{r})_r$ to $\mainfuncmin$ in terms of $r$. That is, we wish to find bounds in terms of $r$ for the parameter:
\begin{equation}
\ubError{r}{\mainset, \mu}{\mainfunc} := \Lasupbound{r} - \mainfuncmin,
\end{equation}
often also denoted $\ubError{r}{}{\mainfunc}$ for simplicity when there is no ambiguity on $K,\mu$.

\subsection{Related work}
\label{SEC_relatedwork}
Bounds on the parameter $\ubError{r}{\mainset,\mu}{\mainfunc}$ have been shown in the literature for several different sets of assumptions on $\mainset, \mu$ and $\mainfunc$. Depending on these assumptions, two main strategies have been employed, which we now briefly discuss.

\medskip
\noindent
\textbf{Algebraic analysis via an eigenvalue reformulation.} 
The first strategy relies on a reformulation of the optimization problem \eqref{EQ_fmin_measures_bounded} as an eigenvalue minimization problem (see \cite{deKlerk2017, Lasserre2010}).
We describe it briefly, in the univariate case $n=1$ only, for  simplicity and since this is the case we need.
Let $\{ p_r \in \R[x]_r: r \in \N\}$ be the (unique) orthonormal basis of $\R[x]$ w.r.t. the inner product $\langle p_i, p_j \rangle = \int_\mainset p_i p_j d\mu$. For each $r \in \mathbb{N}$, we then define the (generalized) \emph{truncated moment matrix} $M_{r,f}$
of $\mainfunc$ by setting 
\begin{equation}
    M_{r,f}(i,j) := \int_{\mainset} p_i p_j f d \mu \quad \text{for } 0 \leq i,j \leq r.
\end{equation}
It can be shown that $\Lasupbound{r} =\lambda_{\min}(M_{r,f})$, 
the smallest eigenvalue of the matrix $M_{r,f}$.
Any bounds on the eigenvalues of $M_{r,f}$ thus immediately translate to bounds on $\Lasupbound{r}$. 


In \cite{deKlerkLaurent2018}, the authors 
determine the exact asymptotic behaviour of $\lambda_{\min}(M_{r,f})$ in the case that $\mainfunc$ is a quadratic polynomial, $\mainset = [-1, 1]$ and $d\mu(x) = (1-x^2)^{-\frac{1}{2}}dx$, known as the Chebyshev measure. Based on this, they  show that $\ubError{r}{}{\mainfunc}=\bigO(1/r^2)$ and  
 extend this result to arbitrary multivariate polynomials $\mainfunc$ on the hypercube $[-1, 1]^n$ equipped with the product measure $d\mu(x) = \prod_i^n (1-x_i)^{-1/2}dx_i$. In addition, they prove that $\ubError{r}{}{\mainfunc} = \Theta(1/r^2)$ for linear polynomials,
which thus shows that in some sense quadratic convergence is the best we can hope for. (This latter result is shown in \cite{deKlerkLaurent2018} for all measures with Jacobi weight on $[-1,1]$).

The  orthogonal polynomials corresponding to the  measure $(1-x^2)^{-1/2}dx$ on $[-1,1]$ are  the  Chebyshev polynomials  of the first kind, denoted by $\Cheby{r}$. They  are well-studied objects (see, e.g., \cite{Rivlin1990}). In particular, they satisfy the following \emph{three-term recurrence} relation
\begin{equation}\label{eq:3term}
    \Cheby{0}(x) = 1, \Cheby{1}(x) = x, \text{ and } \Cheby{r+1}(x) = 2x\Cheby{r}(x) - \Cheby{r-1}(x) \text{ for } r \geq 1.
\end{equation}
This imposes a large amount of structure on the matrix $M_{r, f}$ when $\mainfunc$ is quadratic, which has been  exploited in \cite{deKlerkLaurent2018} to obtain information on its smallest eigenvalue.

The main disadvantage of the eigenvalue strategy is that it requires the moment matrix of $\mainfunc$ to have a closed form expression which is sufficiently structured so as to allow for an analysis of its eigenvalues. Closed form expressions for the entries of the matrix $M_{r,f}$ are known only for  special sets $\mainset$, such as the interval $[-1,1]$, the unit ball,  the unit sphere, or  the simplex, and only with respect to certain measures.

However, as we will see in this paper,  the convergence analysis from \cite{deKlerkLaurent2018} in $\bigO(1/r^2)$ for the interval $[-1,1]$ equipped with  the Chebyshev measure, 
can be transported 
to a large class of compact sets, such as the interval $[-1,1]$ with more general measures, the ball, the simplex, and `ball-like' convex bodies. 

\medskip 
\noindent
\textbf{Analysis via the construction of feasible solutions.}
A second strategy to bound the convergence rate of the parameters $\ubError{r}{}{\mainfunc}$ is to construct explicit sum-of-squares density functions $\Laspolye{r} \in \Sigma_r$ for which the integral $\int_\mainset q_r \mainfunc d\mu$ is close to $\mainfuncmin$. In contrast to the previous strategy, such constructions will only yield upper bounds on $\ubError{r}{}{\mainfunc}$.

As noted earlier, the integral $\int_\mainset \mainfunc d\nu$ may be minimized by selecting the probability measure  $\nu = \delta_a$,  the Dirac measure  at a global minimizer $a$ of $\mainfunc$ on $\mainset$. When the reference measure $\mu$ is the Lebesgue measure, it thus intuitively seems sensible to consider sum-of-squares densities $\Laspolye{r}$ that approximate the Dirac delta in some way. 

This approach is followed in \cite{deKlerk2017}. There, the authors consider truncated Taylor expansions of the Gaussian function $e^{-t^2/2\sigma}$, which they use to define the sum-of-squares polynomials
\begin{equation}
    \phi_{r}(t) = \sum_{k=0}^{2r} \frac{1}{k!}\bigg(\frac{-t^{2}}{2 \sigma}\bigg)^k\in \Sigma_{2r}\quad \text{ for } r\in \N.
\end{equation}
Setting $q_r(x) \sim \phi_{r}(||x-a||)$ for carefully selected standard deviation $\sigma = \sigma(r)$, 
they show
  that $\int_{\mainset} \mainfunc(x) \Laspolye{r}(x) dx - \mainfunc(a) = \bigO(1/\sqrt{r})$ when $\mainset$ satisfies a minor geometrical assumption (Assumption~\ref{ASSU_one} below), which holds, e.g., if $\mainset$ is a convex body or if it is star-shaped with respect to a ball.

In  subsequent work \cite{deKlerkLaurent2017}, the authors show that if $\mainset$ is assumed to be a convex body, then a bound in $\bigO(1/r)$ may be obtained by setting $\Laspolye{r} \sim \phi_r(f(x))$. As explained in \cite{deKlerkLaurent2017}, the sum-of-squares density $q_r$ in this case can be seen as an approximation of the Boltzman density function for $\mainfunc$, which plays an important role in simulated annealing.

The advantage of this second strategy seems to be its applicability to a broad class of sets $\mainset$ with respect to the natural Lebesgue measure. This generality, however, is so-far offset by significantly weaker bounds on $\ubError{r}{}{\mainfunc}$.
Another main contribution of this paper will be to show improved bounds on  $\ubError{r}{}{\mainfunc}$ for this broad class of sets $K$.

\medskip
 
\noindent
\textbf{Analysis for the hypersphere.} 
Tight results are known for polynomial minimization on the unit sphere $S^{n-1}=\{x\in\R^n: \sum_ix_i^2=1\}$, equipped with the Haar surface measure.  Doherty and Wehner \cite{Doherty_Wehner2012} have shown a convergence rate in 
$\bigO(1/r)$, by  using harmonic analysis on the sphere and connections to quantum information theory.
In the very recent work \cite{deKlerkLaurent2019}, the authors show an improved  convergence rate in $\bigO(1/r^2)$, by using a reduction to the  case of the interval  $ [-1,1]$ and the above mentioned convergence rate in $\bigO(1/r^2)$ for this case. This reduction is based on   replacing $f$ by an easy (linear) upper estimator. This idea was  already exploited in  \cite{deKlerk2017,deKlerkLaurent2017} (where a quadratic upper estimator was used) and we will also exploit it in this paper.

\subsection{Our contribution}
The contribution of this paper is showing improved bounds on the convergence rate of the parameters $\ubError{r}{\mainset,\mu}{\mainfunc}$ for a wide class of sets $K$ and measures $\mu$. It is twofold.

Firstly, 
we extend the known bound from  \cite{deKlerkLaurent2018}  in $\bigO(1/r^2)$ for the hypercube $[-1,1]^n$ equipped with the Chebyshev measure, to a wider class of convex bodies. Our results hold for the ball $\unitball$, the simplex $\simplex{n}$, and `ball-like' convex bodies (see Definition \ref{DEF:round}) equipped with the Lebesgue measure. For the ball and hypercube, they further hold for a wider class of measures; namely for the measures given by
\begin{equation}
	w_\lambda(x)dx := (1-\|x\|^2)^\lambda dx \quad (\lambda \geq 0)
\end{equation} 
on the ball, and the measures
\begin{equation}
	\widehat{w_\lambda}(x)dx := \prod_{i=1}^n (1-x_i^2)^\lambda dx \quad (\lambda \geq -\frac{1}{2})
\end{equation} 
on the hypercube. Note that for the hypercube, setting $\lambda = -\frac{1}{2}$ yields the Chebyshev measure, and that for both the ball and the hypercube, setting $\lambda = 0$ yields the Lebesgue measure. 
The rate $\bigO(1/r^2)$ also holds for any compact $K$ equipped with the Lebesgue measure under the assumption of existence of a global minimizer in the interior of $K$. These results are presented in Section~\ref{SEC_SpecialConvexBodies}.

Secondly,  we improve the known bounds in $\bigO(1/\sqrt{r})$ and $\bigO(1/r)$ for general compact sets (under Assumption~\ref{ASSU_one}) and convex bodies equipped with the Lebesgue measure,  established in \cite{deKlerk2017}, \cite{deKlerkLaurent2017}, respectively. For general compact sets, we prove a bound in $\bigO(\log r / r)$,  and for convex bodies we show a bound in $\bigO(\log^2 r / r^2)$. 
These results are exposed in Section~\ref{SEC_GeneralSets}.

\begin{table}
\centering
\def\arraystretch{1.3}
\begin{tabular}{|c|c|c|c|}
\hline
$\mainset \subseteq \R^n$ compact & $\ubError{r}{\mainset, \mu}{\mainfunc}$ & $\mu$  & reference \\
\hline
General  & $o(1)$ & Borel & \cite{Lasserre2010}   \\
Assumption \ref{ASSU_one} &$\bigO(1/\sqrt{r})$ & Lebesgue & \cite{deKlerk2017} \\
Convex body  & $\bigO(1/r)$ & Lebesgue  & \cite{deKlerkLaurent2017} \\
Hypersphere  & $\bigO(1/r)$ & Haar  & \cite{Doherty_Wehner2012} \\
Hypersphere & $\bigO(1/r^2)$ & Haar  & \cite{deKlerkLaurent2019} \\
Hypercube  & $\bigO(1/r^2)$ & Chebyshev  & \cite{deKlerkLaurent2018} \\
\hline
Hypercube &\multirow{5}{*}{$\bigO(1/r^2)$} & $\widehat{w_\lambda}(x)dx ~~ (\lambda \geq -1/2)$ & Thm. \ref{THM_general_box} \\
Ball &  & $w_\lambda(x)dx ~~ (\lambda \geq 0)$ & Thm. \ref{THM_general_ball} \\
Simplex &  & Lebesgue & Thm. \ref{THM_simplex} \\
Ball-like convex body &  & Lebesgue & Thm. \ref{THM_squeeze} \\
Global minimizer in the interior &  & Lebesgue & Thm. \ref{THM_interiorglobalminimizer} \\
\hline
Assumption \ref{ASSU_one} & $\bigO(\log r / r)$ & \multirow{2}{*}{Lebesgue} & Thm. \ref{THM_bound_interior_cone} \\
Convex body & $\bigO(\log^2 r / r^2)$ &  & Thm. \ref{THM_bound_convex_body} \\
\hline
\end{tabular}
\caption{Known and new convergence rates for the Lasserre hierarchy of upper bounds $\ubError{r}{\mainset, \mu}{\mainfunc}$.}
\label{TAB_convergence_rates}
\end{table}

For our results in Section~\ref{SEC_SpecialConvexBodies}, we will use several tools that will enable us to reduce to the case of the interval $[-1,1]$ equipped with the Chebyshev measure. These tools are presented in \mbox{Sections \ref{SEC_Preliminaries} and \ref{SEC_SpecialConvexBodies}}.
 They include: (a) replacing $K$ by an affine linear image of it (Section~\ref{SEC:Linear}); (b) replacing $f$ by an upper estimator (easier to analyze, obtained via Taylor's theorem) (Section~\ref{SEC:Upper}); (c) transporting results between two comparable weight functions on $K$ and between two convex sets $K,\widehat K$ which look locally the same in the neighbourhood of a global minimizer \mbox{(Sections~\ref{SEC:weights}, \ref{SEC:localsimilarity})}. In particular, the result of Proposition~\ref{PROP_differentmeasure} will play a key role in our treatment.

To establish our results in Section~\ref{SEC_GeneralSets} we will follow the second strategy sketched above, namely we will define suitable sum-of-squares polynomials that approximate well the Dirac delta at a global minimizer.
However, instead of using truncations of the Taylor expansion of  the Gaussian function or of the Boltzman distribution as was done in \cite{deKlerk2017}, \cite{deKlerkLaurent2017}, we will now use the so-called needle polynomials from \cite{Kroo1992} (constructed from the Chebyshev polynomials, see Section~\ref{SEC_needles}).
In Table~\ref{TAB_convergence_rates} we provide an overview of both known and new results.

Finally, we illustrate some of the results in Sections \ref{SEC_SpecialConvexBodies} and \ref{SEC_GeneralSets} with numerical examples in Section \ref{SEC_NumericalExperiments}.
\section{Preliminaries}
    \label{SEC_Preliminaries}
    In this section, we first introduce some notation that we will use throughout the rest of the paper and recall some basic terminology and results about convex bodies. We then show that the error $\ubError{r}{}{\mainfunc}$ is invariant under nonsingular affine transformations of $\R^n$. Finally, we introduce the notion of \emph{upper estimators} for $\mainfunc$. Roughly speaking, this tool will allow us to replace $\mainfunc$ in the analysis of $\ubError{r}{}{\mainfunc}$ by a simpler function (usually a quadratic, separable polynomial). We will make use of this extensively in both Section~\ref{SEC_SpecialConvexBodies} and Section~\ref{SEC_GeneralSets}.

\subsection{Notation}
For $x, y \in \R^n$, 
$\langle x, y \rangle$ denotes  the standard inner product 
and  $\|x\|^2 = \langle x, x \rangle$  the corresponding norm. We write $\ball{\rho}{c} := \{x \in \R^n : \|x - c\| \leq \rho \}$ for the $n$-dimensional ball of radius $\rho$ centered at $c$. 
When $\rho = 1$ and $c = 0$, we also write $\unitball := \ball{1}{0}$.

Throughout, $\mainset \subseteq \R^n$ is always a compact set with non-empty interior, and $\mainfunc$ is an $n$-variate 
polynomial. We let $\nabla \mainfunc (x)$ (resp., $\nabla^2f(x)$)  denote the gradient (resp., the Hessian) of $\mainfunc$ at $x \in \R^n$, and introduce the parameters 
\begin{equation}
\label{EQ_gradHessconsts}
\gradConst{\mainfunc}{\mainset} := \max_{x \in \mainset} \|\nabla \mainfunc (x)\| \quad \text{ and } \quad \HessConst{\mainfunc}{\mainset} := \frac{1}{2}\max_{x \in \mainset} \|\nabla^2 \mainfunc (x)\|.
\end{equation}
Whenever we write an expression of the form
\[
``\ubError{r}{}{\mainfunc} = \bigO(1/r^2)",
\]
we mean that there exists a constant $c > 0$ such that $\ubError{r}{}{\mainfunc}\leq c/r^2$ for all $r \in \N$, where $c$ depends only on $\mainset, \mu$, and the parameters $\gradConst{\mainfunc}{\mainset}, \HessConst{\mainfunc}{\mainset}$. 
Some of our results are obtained by embedding $\mainset$ into a larger set $\biggerset \subseteq \R^n$. 
If this is the case, then $c$ may depend on $\gradConst{\mainfunc}{\biggerset}, \HessConst{\mainfunc}{\biggerset}$ as well.
If there is an additional dependence of $c$ on the global minimizer $a$ of $\mainfunc$ on $\mainset$, we will make this explicit by using the notation $``\bigO_a"$.

\subsection{Convex bodies}
Let $\mainset \subseteq \R^n$ be a convex body, i.e., a compact, convex set with non-empty interior. 
We say $v \in \R^n$ is an (inward) normal of $\mainset$ at $a \in \mainset$ if $\langle v, x-a \rangle \geq 0$ holds for all $x \in \mainset$. We refer to the set of all normals of $\mainset$ at $a$ as the normal cone, and write
\begin{equation}
    \label{EQ_normalconedefinition}
    \normalcone{\mainset}{a} := \{ v \in \R^n : \langle v, x - a \rangle \geq 0 \text{ for all } x \in \mainset\}.
\end{equation}
We will make use of the following basic result.

\begin{lem}[e.g., {\cite[Prop. 2.1.1]{BorweinLewis}}]
\label{LEM_first_order_condition}
Let $\mainset$ be a convex body and let $g : \R^n \rightarrow \R$ be a continuously differentiable function with local minimizer $a \in \mainset$. Then $\nabla g (a) \in \normalcone{\mainset}{a}$.
\end{lem}

\begin{proof}
Suppose not. Then, by definition of $\normalcone{\mainset}{a}$, there exists an element  $y \in \mainset$ such that  $\langle \nabla g(a), y - a \rangle < 0$. Expanding the definition of the gradient this means that
\begin{equation}
    0 > \langle \nabla g(a), y - a \rangle = \lim_{t \downarrow  0}\frac{g(ty + (1-t) a) - g(a)}{t},
\end{equation}
which implies  $g(ty + (1-t) a) < g(a)$ for all $t > 0$ small enough.  But $ty + (1-t) a \in \mainset$ by convexity, contradicting the fact that $a$ is a local minimizer of $g$ on $\mainset$.
\qedMP\end{proof}

The set $\mainset$ is \emph{smooth} if it has a unique unit normal $\normal{a}$ at each boundary point $a \in \partial\mainset$. In this case, we denote by $T_a\mainset$ the (unique) hyperplane tangent to $\mainset$ at $a$, defined by the equation $\langle x-a, \normal{a} \rangle = 0$.

For $k \geq 1$, we say $\mainset$ is of class $C^k$ if there exists a convex function $\Psi \in C^k(\R^n, \R)$ such that $\mainset = \{ x \in \R^n : \Psi(x) \leq 0\}$ and $\partial \mainset =  \{x \in \R^n : \Psi(x) = 0\}$. 
If $\mainset$ is of class $C^k$ for some $k \geq 1$, it is automatically smooth in the above sense.

We refer, e.g., to \cite{bonnesenConvexBodies} for a general reference on convex bodies.

\subsection{Linear transformations}\label{SEC:Linear}

Suppose that $\phi : \R^n \rightarrow \R^n$ is a nonsingular affine transformation, given by $\phi(x) = Ux + c$. If $\Laspoly$ is a sum-of-squares density function w.r.t. the Lebesgue measure on $\phi(\mainset)$, then we have
\begin{equation}
\int_{\phi(\mainset)} \Laspoly(y) \mainfunc(\phi^{-1}(y)) dy = |\det U| \cdot \int_{\mainset} \Laspoly(\phi(x)) \mainfunc(x)dx\quad \text{ and }
\end{equation}
\begin{equation}
1 = \int_{\phi(\mainset)} \Laspoly(y) dy =  |\det U| \cdot \int_{\mainset} \Laspoly(\phi(x)) dx.
\end{equation}
As a result, the polynomial $\widehat\Laspoly := (\Laspoly \circ \phi) / \int_\mainset \Laspoly(\phi(x))dx = (\Laspoly \circ \phi) \cdot |\det U|$ is a sum of squares density function w.r.t. the Lebesgue measure on $\mainset$. It has the same degree as $q$, and it satisfies
\begin{equation}
\int_{\mainset} \widehat{\Laspoly}(x) \mainfunc(y)dx = \int_{\phi(\mainset)} \Laspoly(x) \mainfunc(\phi^{-1}(y)) dx.
\end{equation}
We have just shown the following.
\begin{lem}
\label{LEM_lineartransformation}
Let $\phi : \R^n \rightarrow \R^n$ be a non-singular affine transformation. Write $g := \mainfunc \circ \phi^{-1}$. Then we have
\begin{equation}
\ubError{r}{\mainset}{\mainfunc} = \ubError{r}{\phi(\mainset)}{g}.
\end{equation}
\end{lem}

\subsection{Upper estimators}\label{SEC:Upper}

Given a point $a \in \mainset$ and two functions $f,g : \mainset \rightarrow \R$, we write $\mainfunc \leq_a g$ if $\mainfunc(a) = g(a)$ and $\mainfunc(x) \leq g(x)$ for all $x \in \mainset$; we then say that   $g$ is an \emph{upper estimator} for $\mainfunc$ on $\mainset$, which is \emph{exact at} $a$. The next lemma, whose easy proof is  omitted,  will be very useful.

\begin{lem}
\label{LEM_upperestimator}
Let $g : \mainset \rightarrow \R$ be an upper estimator for $\mainfunc$, exact at one of its global minimizers  on $\mainset$. Then we have  $\ubError{r}{}{\mainfunc} \leq \ubError{r}{}{g}$ for all $r \in \N$.
\end{lem}


\begin{rem}
\label{REM_reducetoestimator}
We make the following observations for future reference.
\begin{itemize}
    \item[1.] Lemma~\ref{LEM_upperestimator} tells us that we may always replace $\mainfunc$ in our analysis by an upper estimator which is exact at one of its global minimizers. This is useful if we can find an upper estimator that is significantly simpler to analyse. 
    
    \item[2.] We may always assume that $\mainfuncmin = 0$, in which case $\mainfunc(x) \geq 0$ for all $x\in \mainset$ and $\ubError{r}{}{\mainfunc} = \Lasupbound{r}$. Indeed, if we consider the function $g$ given by $g(x) = \mainfunc(x) - \mainfuncmin$, then $g_{\min, \mainset} = 0$, and for every density function $q$ on $\mainset$, we have
        \begin{equation}
            \int_\mainset g(x)\Laspoly(x)d\mu(x) = \int_\mainset \mainfunc(x)\Laspoly(x)d\mu(x) - \mainfuncmin,
        \end{equation}
        showing that $\ubError{r}{}{\mainfunc} = \ubError{r}{}{g} = g^{(r)}$ for all $r \in \N$.
\end{itemize}
\end{rem}

In the remainder of this section, we derive some general upper estimators based on the following variant of Taylor's theorem for multivariate functions.
\begin{thm}[Taylor's theorem]
\label{THM_Taylor}
For  $f  \in C^2(\R^n, \R)$ and $a \in \mainset$  we have
\begin{equation}
    \label{EQ_Taylor_bound}
    f(x) \leq f(a) + \langle \nabla f(a), x-a \rangle + \HessConst{\mainset}{f} \|x-a\|^2 \quad \text{for all } x \in \mainset,
\end{equation}
where $\gamma_{K,f}$ is the constant from \eqref{EQ_gradHessconsts}.
\end{thm}

\begin{lem}
\label{LEM_quadraticupperestimator}
Let $a \in \mainset$ be a global minimizer of $\mainfunc$ on $\mainset$. Then $\mainfunc$ has an upper estimator $g$ on $\mainset$ which is exact at $a$ and satisfies the following properties:
\begin{enumerate}
    \item[(i)] $g$ is a quadratic, separable polynomial.
    \item[(ii)] $g(x) \geq \mainfunc(a) + \HessConst{\mainset}{\mainfunc} \|x-a\|^2$ for all $x \in \mainset$.
    \item[(iii)] If $a \in \interior{\mainset}$, then $g(x) \leq \mainfunc(a) + \HessConst{\mainset}{\mainfunc} \|x-a\|^2$ for all $x \in \mainset$.
\end{enumerate}
\end{lem}
\begin{proof}
Consider the function $g$ defined by 
\begin{equation}\label{EQ:g}
    g(x) := \mainfunc(a) + \langle \nabla \mainfunc(a), x-a \rangle + \HessConst{\mainset}{\mainfunc} \|x-a\|^2,
\end{equation}
which is an upper estimator of $\mainfunc$ exact at $a$ by Theorem~\ref{THM_Taylor}. As we have 
$\|x-a\|^2 = \sum_{i}^n (x_i - a_i)^2$, $g$ is indeed a quadratic, separable polynomial.

As $a$ is a global minimizer of $\mainfunc$ on $\mainset$, we know by Lemma~\ref{LEM_first_order_condition} that $\nabla \mainfunc(a) \in \normalcone{\mainset}{a}$. 
This means that $\langle \nabla \mainfunc(a), x-a \rangle \geq 0$ for all $x \in \mainset$, which proves the second property.

If $a \in \interior{\mainset}$, we must have $\nabla \mainfunc (a) = 0$, and the third property follows.
\qedMP\end{proof}
In the special case that $\mainset$ is a ball and $\mainfunc$ has a global minimizer $a$ on the boundary of $\mainset$, we have an upper estimator for $\mainfunc$, exact at $a$, which is a linear polynomial.

\begin{lem}
\label{LEM_linear_upperestimator}
Assume that $f(a) = f_{\min, \ball{\rho}{c}}$ for some $a \in \partial \ball{\rho}{c}$. Then there exists a linear polynomial $g$ with $\mainfunc \leq_a g$ on $\ball{\rho}{c}$.
\end{lem}

\begin{proof}
Write $K=\ball{\rho}{c}$ and $\gamma=\gamma_{K,f}$ for simplicity. 
In view of Lemma~\ref{LEM_quadraticupperestimator}, we have $f(x)\le g(x)$ for all $x\in K$, where $g$ is the quadratic polynomial from relation  \eqref{EQ:g}.
Since $a\in \partial K$ is a global minimizer of $\mainfunc$ on $K$, we have $\nabla f(a) \in \normalcone{K}{a}$ by Lemma~\ref{LEM_first_order_condition}, and thus  $\nabla f(a) = \lambda (c-a)$ for some $\lambda \ge 0$. Therefore we have
$$\langle \nabla f(a),x-a\rangle=\langle \lambda (c-a), x-a\rangle=\lambda \rho^2+ \lambda \langle x-c,c-a\rangle.$$
On the other hand, for any $x\in K$ we have
$$\|x-a\|^2 = \|x-c\|^2 + \|c-a\|^2 + 2 \langle x-c, c-a \rangle \le 2\rho^2+2\langle x-c,c-a\rangle.$$
Combining these facts we get
$$f(x)\le g(x)\le f(a)+ (\lambda+2\gamma) (\rho^2+ \langle x-c,c-a\rangle)=: h(x).$$
So $h(x)$ is a linear upper estimator of $f$ with $h(a)=f(a)$, as desired.
\qedMP
\end{proof}

\begin{rem}
\label{REM_gradientinnormalcone}
As can be seen from the above proof, the assumption in Lemma~\ref{LEM_linear_upperestimator} that $a \in \partial \mainset = \partial \ball{\rho}{c}$ is a global minimizer of $\mainfunc$ on $\mainset$ may be replaced by the weaker assumption that $\nabla \mainfunc(a) \in \normalcone{\mainset}{a}$.
\end{rem}
Finally, we give a very simple upper estimator, which will be used in Section~\ref{SEC_GeneralSets}.
\begin{lem}
\label{LEM_Lipschitzupperestimator}
Recall the constant $\gradConst{\mainset}{\mainfunc}$ from \eqref{EQ_gradHessconsts}. Let $a$ be a global minimizer of $\mainfunc$ on $\mainset$. Then we have
\begin{equation}
    \mainfunc(x) \leq_a \mainfunc(a) + \gradConst{\mainset}{\mainfunc} \|x - a \| \quad \text{for all } x \in \mainset.
\end{equation}
\end{lem}
\section{Special convex bodies}
    \label{SEC_SpecialConvexBodies}
    In this section we extend the bound $\bigO(1/r^2)$  from \cite{deKlerkLaurent2018}  on $\ubError{r}{\mainset, \mu}{\mainfunc}$, when $\mainset = [-1, 1]^n$ is equipped with  the Chebyshev measure  $d\mu(x) = \prod_{i=1}^n (1-x_i^2)^{-\frac{1}{2}}dx_i$, to a broader class of convex bodies $\mainset$ and reference measures $\mu$.

First, we show that, for the hypercube $K=[-1,1]^n$, we still have $\ubError{r}{\mainset, \mu}{\mainfunc} = \bigO(1/r^2)$ for all $f$ and all measures of the form $d\mu(x) = \prod_{i=1}^n (1-x_i^2)^{\lambda}dx_i$ with $\lambda > -1/2$. Previously this was only known to be the case when $\mainfunc$ is a linear polynomial. Note that, for $\lambda = 0$, we obtain the Lebesgue measure on $[-1, 1]^n$.
Next, we use this result to show that $\ubError{r}{\unitball, \mu}{\mainfunc} = \bigO(1/r^2)$ for all measures $\mu$ on the unit ball $\unitball$ of the form $d\mu(x) = (1-||x||^2)^{\lambda}dx$ with $\lambda \geq 0$. 
We apply this result to also obtain $\ubError{r}{\mainset, \mu}{\mainfunc} = \bigO(1/r^2)$ when $\mu$ is the Lebesgue measure and $\mainset$ is a `ball-like' convex body, meaning it has inscribed and circumscribed tangent balls at all boundary points (see Definition~\ref{DEF:round} below). 
The primary new tool we use to obtain these results is \mbox{Proposition~\ref{PROP_differentmeasure}}, which tells us that the behaviour of $\ubError{r}{\mainset, \mu}{\mainfunc}$ essentially only depends on the \emph{local} behaviour of $\mainfunc$ and $\mu$ in a neighbourhood of a global minimizer $a$ of $\mainfunc$ on $\mainset$.

\subsection{Measures and weight functions}
\label{SEC:weights}
A function  $\measureweight{} : \interior{\mainset} \rightarrow \R_{> 0}$ is a \emph{weight function} on $\mainset$ if it is continuous and satisfies $0 < \int_\mainset \measureweight{}(x) dx < \infty$. A weight function $\measureweight{}$ gives rise to a measure $\lambdameasure{\measureweight{}}$ on $\mainset$ defined by $d\lambdameasure{\measureweight{}}(x) := \measureweight{}(x)dx$. 
We note that if $\mainset \subseteq \biggerset$, and $\widehat{\measureweight{}}$ is a weight function on $\biggerset$, it can naturally be interpreted as a weight function on $\mainset$ as well, by simply restricting its domain (assuming $\int_\mainset \widehat{\measureweight{}}(x) dx > 0$). In what follows we will implicitly make use of this fact.

\begin{defn}
Given two weight functions $\measureweight{}, \widehat{\measureweight{}}$ on $\mainset$ and a point $a \in \mainset$, we say that $\widehat{\measureweight{}} \measureleq{a} \measureweight{}$ on $\mainset$ if there exist constants $\epsilon, m_{a} > 0$ such that
\begin{equation}
\label{EQ_measurelocaldominance}
 m_{a}\widehat{\measureweight{}}(x) \leq \measureweight{}(x) \text{ for all } x \in \ball{\epsilon}{a} \cap \interior{\mainset}.
\end{equation}
If the constant $m_{a}$ can be chosen uniformly, i.e., if there exists a constant $m > 0$ such that
\begin{equation}
\label{EQ_measuredominance}
    m\widehat{\measureweight{}}(x) \leq \measureweight{}(x) \text{ for all } x \in \interior{\mainset},
\end{equation}
then we say that $\widehat{\measureweight{}} \measureleq{} \measureweight{}$ on $\mainset$.
\end{defn}

\begin{rem}
\label{REM_weightdominance}
We note the following facts for future reference:
\begin{enumerate}
	\item[(i)] As weight functions are continuous on the interior of $\mainset$ by definition, we always have $\widehat{\measureweight{}} \measureleq{a} \measureweight{}$ if $a \in \interior{\mainset}$. 
	\item[(ii)] If $\measureweight{}$ is bounded from below, and $\widehat{\measureweight{}}$ is bounded from above on $\interior{\mainset}$, then we automatically have $\widehat{\measureweight{}} \measureleq{} \measureweight{}$.
\end{enumerate}
\end{rem}

\subsection{Local similarity}
\label{SEC:localsimilarity}
Assuming that the global minimizer $a$ of $\mainfunc$ on $\mainset$ is unique, sum-of-squares density functions $\Laspoly$ for which the integral $\int_{\mainset} \Laspoly(x) \mainfunc(x) d \mu(x)$ is small should in some sense approximate the Dirac delta function centered at $a$. 
With this in mind, it seems reasonable to expect that the quality of the bound $\Lasupbound{r}$ depends in essence only on the \emph{local} properties of $\mainset$ and $\mu$ around $a$. 
We formalize this intuition here.

\begin{defn}
Suppose $\mainset \subseteq \biggerset \subseteq \R^n$. Given $a \in \mainset$, we say that $\mainset$ and $\biggerset$ are \emph{locally similar} at $a$, which we denote by   $\mainset \subseteq_a \biggerset$,  if there exists  $\epsilon > 0$ such that 
\begin{equation}
        \label{EQ_localsimilarity}
        \ball{\epsilon}{a} \cap \mainset = \ball{\epsilon}{a} \cap \biggerset.
\end{equation}
Clearly,  $\mainset \subseteq_a \biggerset$ for any point $a\in \interior K$.
\end{defn}
Figure~\ref{FIG_locallysimilar} depicts some examples of locally similar sets. 

\begin{figure}
    \centering
    \begin{tikzpicture}[x=2cm,y=2cm]

	\draw[thick] (0,0) -- (1, 0) -- (1, 1) -- (0,1) -- (0,0);
	\draw[thick] (0,0) -- (1.5, 0) -- (1.5, 1.5) -- (0,1.5) -- (0,0);

	\coordinate (a) at (0,0) ;   
    \fill[red] (a) circle (2pt);
    \draw[thick, dashed] (a) circle (1cm);
    
    \draw[] (1,1) node[below right] {$\mainset$};
    \draw[] (1.5,1.5) node[below right] {$\biggerset$};

	\begin{scope}
		\clip (0,0) -- (1, 0) -- (1, 1) -- (0,1) -- (0,0);

		\fill[opacity=0.2] (0, 0) circle (1cm);
	\end{scope}
\end{tikzpicture}
    \begin{tikzpicture}[x=2cm,y=2cm]

	\coordinate (a) at (0.75,0);

  	\begin{scope}[shift={(a)}]
  		\draw[thick] (-0.75,0) -- (-0.75, 0.6) -- (0.75, 0.6) -- (0.75,0) -- (-0.75,0);
  		\draw[thick] (0.75, 0.6) node[below right] {$\mainset$};
    	\draw[thick] (0.9,1) node[] {$\biggerset$};
	\end{scope}
  
  	\begin{scope}[shift={(a)}]
	\draw[thick] (-1.2, 0) -- (1.2,0);
  	\clip (-1.3, 0) -- (-1.3, 1.3) -- (1.3, 1.3) -- (1.3, 0) -- cycle;
  	\draw[thick] (0,0) circle (2.4cm);
	\end{scope}

    \draw[dashed, thick] (a) circle (1cm);

	\begin{scope}[shift={(a)}]
		\clip (-0.75,0) -- (-0.75, 0.5) -- (0.75, 0.5) -- (0.75,0) -- (-0.75,0);
		\fill[opacity=0.2] (0, 0) circle (1cm);
	\end{scope}
	
    \fill[red] (a) circle (2pt);    
\end{tikzpicture}
    \caption{Some examples of sets $\mainset, \biggerset$ for which $\mainset \subseteq_a \biggerset$. The red dot indicates the point $a$, and the gray area indicates $\ball{\epsilon}{a} \cap \mainset$.}
    \label{FIG_locallysimilar}
\end{figure}
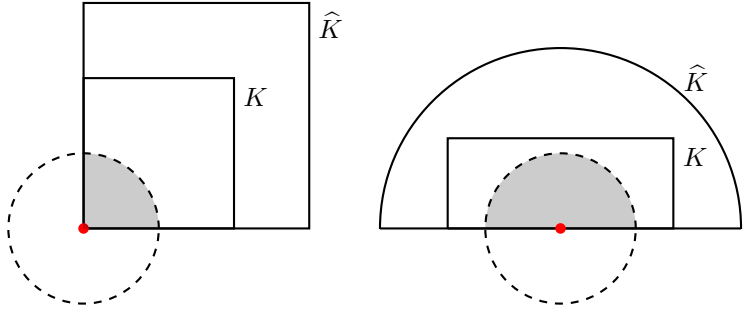

\begin{prop}
\label{PROP_differentmeasure}
Let $\mainset \subseteq \biggerset \subseteq \R^n$, let $a \in \mainset$ be a global minimizer of $\mainfunc$ on $\mainset$ and assume $K\subseteq _a \biggerset$.
Let $\measureweight{}, \widehat{\measureweight{}}$ be two weight functions on $\mainset, \biggerset$, respectively. Assume that $\widehat{\measureweight{}}(x) \geq \measureweight{}(x)$ for all $x \in \interior{\mainset}$, and that $\widehat{\measureweight{}} \measureleq{a} \measureweight{}$. 
Then there exists an upper estimator $g$ of $\mainfunc$ on $\biggerset$ which is exact at $a$ and satisfies
\begin{equation}
   \ubError{r}{\mainset, \measureweight{}}{g} \leq \frac{2}{m_a}  \ubError{r}{\biggerset, \widehat{\measureweight{}}}{g}
\end{equation}
for all $r \in \N$ large enough. Here $m_a > 0$ is 
the constant defined by \eqref{EQ_measurelocaldominance}.
\end{prop}
Recall that if $g$ is an upper estimator for $\mainfunc$ which is exact at one of its global minimizers, we then have $\ubError{r}{\mainset, \measureweight{}}{\mainfunc} \leq \ubError{r}{\mainset, \measureweight{}}{g}$ by Lemma~\ref{LEM_upperestimator}. Proposition~\ref{PROP_differentmeasure} then allows us to bound $\ubError{r}{\mainset, \measureweight{}}{\mainfunc}$ in terms of $\ubError{r}{\biggerset, \widehat{\measureweight{}}}{g}$. For its proof, we first need the following lemma.

\begin{lem}
Let $a \in \mainset$, and assume that $\mainset \subseteq_a \biggerset$. Then any normal vector of $\mainset$ at $a$ is also a normal vector of $\biggerset$. That is, $\normalcone{\mainset}{a} \subseteq \normalcone{\biggerset}{a}.$
\label{LEM_normalconelocalsimilarity}
\end{lem}

\begin{proof}
Let $v \in \normalcone{\mainset}{a}$. Suppose for  contradiction that $v \not \in \normalcone{\biggerset}{a}$. Then, by definition of the normal cone, there exists  $y \in \biggerset$ such that $\langle v, y-a \rangle < 0$. 
As $\mainset \subseteq_a \biggerset$, there exists  $\epsilon > 0$ for which $\mainset \cap \ball{a}{\epsilon} = \biggerset \cap \ball{a}{\epsilon}$. 
Now choose $1> \eta > 0$ small enough such that $y' := \eta y + (1-\eta)a \in \ball{a}{\epsilon}$ . Then, by convexity, we have $y' \in \biggerset \cap \ball{a}{\epsilon} = \mainset \cap \ball{a}{\epsilon}$. Now, we have
$\langle v, y'-a \rangle = \eta \langle v, y-a\rangle < 0.$ 
But, as $y' \in K$,  this contradicts the assumption that $v \in \normalcone{\mainset}{a}$. 
\qedMP \end{proof}

\begin{proof}[of Proposition~\ref{PROP_differentmeasure}]
For simplicity, we assume here $\mainfunc(a) = 0$, which is without loss of generality by Remark~\ref{REM_reducetoestimator}.
Consider the quadratic polynomial $g$ from \eqref{EQ:g}: 
\begin{equation}
	g(x) = \langle \nabla \mainfunc(a), x - a \rangle + \gamma ||x-a||^2,
\end{equation}
where $\gamma := \HessConst{\biggerset}{\mainfunc}$ is  defined in \eqref{EQ_gradHessconsts}.
By Taylor's theorem (Theorem~\ref{THM_Taylor}), we have that $g(x) \geq \mainfunc(x)$ for all $x \in \biggerset$, and clearly $g(a) = \mainfunc(a)$. That is, $g$ is an upper estimator for $\mainfunc$ on $\biggerset$, exact at $a$ (cf. Lemma \ref{LEM_quadraticupperestimator}).
We proceed to show that 
\begin{equation}
   \ubError{r}{\mainset, \measureweight{}}{g} \leq \frac{2}{m_a}  \ubError{r}{\biggerset, \widehat{\measureweight{}}}{g}.
\end{equation}
We start by selecting a degree $2r$ sum-of-squares polynomial $\widehat{\Laspolye{r}}$ satisfying
$$\int_{\biggerset} \widehat{\Laspolye{r}}(x) \widehat{\measureweight{}}(x) dx = 1 \quad \text{ and }\quad
    \int_{\biggerset} g(x) \widehat{\Laspolye{r}}(x) \widehat{\measureweight{}}(x) dx = \ubError{r}{\biggerset, \widehat{\measureweight{}}}{g}. $$
We may then rescale $\widehat{\Laspolye{r}}$ to obtain a density function $\Laspolye{r} \in \Sigma_r$ on $\mainset$ w.r.t. $\measureweight{}$ by setting
\begin{equation}
    \Laspolye{r} := \frac{\widehat{\Laspolye{r}}}{\int_\mainset \widehat{\Laspolye{r}}(x) \measureweight{}(x) dx}.
\end{equation}
By assumption,  $\measureweight{}(x) \leq \widehat{\measureweight{}}(x)$ for all $x\in \interior K$. Moreover,  $g(x) \geq \mainfunc(a) = 0$ for all $x \in \interior{\mainset}$.
This implies that
\begin{equation}
    \ubError{r}{\mainset, \measureweight{}}{g} \leq \int_{\mainset} g(x) \Laspolye{r}(x)  \measureweight{}(x) dx 
    \leq \frac{\int_{\biggerset}  g(x) \widehat{\Laspolye{r}}(x) \widehat{\measureweight{}}(x) dx}{\int_\mainset \widehat{\Laspolye{r}}(x) \measureweight{}(x) dx}
    = \frac{\ubError{r}{\biggerset, \widehat{\measureweight{}}}{g}}{\int_\mainset \widehat{\Laspolye{r}}(x) \measureweight{}(x) dx}
\end{equation}
and thus it suffices to show that $\int_\mainset \widehat{ \Laspolye{r}}(x) \measureweight{}(x)dx \geq \frac{1}{2}m_a$.
The key to proving this bound is the following lemma, which tells us that optimum sum-of-squares densities should assign rather high weight to the ball $\ball{\epsilon}{a}$ around $a$.
\begin{lem}
\label{LEM_ballweight}
Let $\epsilon > 0$. Then, for any $r \in \N$, we have
\begin{equation}
    \int_{\ball{\epsilon}{a} \cap \biggerset} \widehat{\Laspolye{r}}(x) \widehat{\measureweight{}}(x)dx \geq 
    1 - \frac{\ubError{r}{\biggerset, \widehat{\measureweight{}}}{g}} {\gamma \epsilon^2}.
\end{equation}
\end{lem}
\begin{proof}
By Lemma~\ref{LEM_first_order_condition}, we have $\nabla \mainfunc(a) \in \normalcone{\mainset}{a}$ and so $\nabla \mainfunc(a) \in \normalcone{\biggerset}{a}$ by Lemma~\ref{LEM_normalconelocalsimilarity}. As a result, we have 
$g(x) \geq \gamma ||x-a||^2$ for all  $ x \in \biggerset$ (cf. Lemma \ref{LEM_quadraticupperestimator}).
In particular, this implies  that $g(x) \geq  \gamma||x-a||^2 \geq  \gamma \epsilon^2$ for all $x \in \biggerset \setminus \ball{\epsilon}{a}$ and so
\begin{align}
    \ubError{r}{\biggerset, \widehat{\measureweight{}}}{g}
    \geq \int_{\biggerset \setminus \ball{\epsilon}{a}} g(x) \widehat{\Laspolye{r}}(x)  \widehat{\measureweight{}}(x) dx 
    &\geq \gamma\epsilon^2 \int_{\biggerset \setminus \ball{\epsilon}{a}} \widehat{\Laspolye{r}}(x) \widehat{\measureweight{}}(x) dx  \\
    &= \gamma\epsilon^2 \bigg(1 - \int_{\ball{\epsilon}{a} \cap \biggerset} \widehat{\Laspolye{r}}(x) \widehat{\measureweight{}}(x) dx\bigg).
\end{align}
The statement now follows from reordering terms.
\qedMP\end{proof} \noindent

As $\mainset \subseteq_a \biggerset$, there exists  $\epsilon_1 > 0$ such that $\ball{\epsilon_1}{a} \cap \mainset = \ball{\epsilon_1}{a} \cap \biggerset$. 
As $\widehat{\measureweight{}} \measureleq{a} \measureweight{}$, there exist $\epsilon_2 > 0$, $m_a > 0$ such that $m_a \widehat{\measureweight{}}(x) \leq \measureweight{}(x)$ for $x \in \ball{\epsilon_2}{a} \cap \interior{\mainset}$. 
Set $\epsilon = \min\{\epsilon_1, \epsilon_2 \}$. Choose $r_0 \in \N$ large enough such that 
$\ubError{r}{\biggerset, \widehat{\measureweight{}}}{g} < {\epsilon^2\gamma\over 2}$ for all $r\ge r_0$,
which is possible since $\ubError{r}{\biggerset, \widehat{\measureweight{}}}{g}$ tends to $0$ as $r\to\infty$.
Then, Lemma~\ref{LEM_ballweight}  yields 
\begin{equation}
    \int_{\ball{\epsilon}{a} \cap \biggerset} \widehat{ \Laspolye{r}}(x) \widehat{\measureweight{}}(x) dx \geq \frac{1}{2}
\end{equation} for all $r \geq r_0$. 
Putting things together yields the desired lower bound:
\begin{equation}
    \int_{\mainset} \widehat{ \Laspolye{r}}(x) {\measureweight{}}(x) dx \geq \int_{\ball{\epsilon}{a} \cap \mainset} \widehat{ \Laspolye{r}}(x) {\measureweight{}}(x) dx \geq m_a\int_{\ball{\epsilon}{a}\cap \biggerset} \widehat{ \Laspolye{r}}(x) \widehat{\measureweight{}}(x)dx \geq \frac{1}{2}m_a.
\end{equation}
for all  $r \geq r_0$.
\qedMP\end{proof}

\begin{cor}
\label{COR_uniformdifferentmeasure}
Let $\mainset \subseteq \biggerset \subseteq \R^n$,  let $a \in \mainset$ be a global minimizer of $\mainfunc$ on $\mainset$, and assume
that $K\subseteq _a \biggerset$.
Let $\measureweight{}, \widehat{\measureweight{}}$ be two weight functions on $\mainset, \biggerset$, respectively. Assume that $\widehat{\measureweight{}}(x) \geq \measureweight{}(x)$ for all $x \in \interior{\mainset}$ and that $\widehat{\measureweight{}} \measureleq{} \measureweight{}$. 
Then there exists an upper estimator $g$ of $\mainfunc$ on $\biggerset$, exact at $a$, such that  
\begin{equation}
   \ubError{r}{\mainset, \measureweight{}}{g} \leq \frac{2}{m}  \ubError{r}{\biggerset, \widehat{\measureweight{}}}{g}
\end{equation}
for all $r \in \N$ large enough. Here $m > 0$ is the constant defined by \eqref{EQ_measuredominance}.
\end{cor}

\subsection{The unit cube}

Here we consider optimization over the hypercube $K=[-1,1]^n$ and  we 
restrict to reference measures on $K$ having a weight function of the form
\begin{equation}
\label{EQ_boxweights}
\widehat {w_{\lambda}}(x):=\prod_{i=1}^nw_{\lambda}(x_i)= \prod_{i=1}^n (1-x_i^2)^{\lambda} 
\end{equation}
with $\lambda>-1$.
The following result is shown  in \cite{deKlerkLaurent2018} on the convergence rate of the bound   $\ubError{r}{\mainset, \widehat{\measureweight{\lambda}}}{f}$ when using the measure  $\widehat{w_{\lambda}}(x) dx $ on $K=[-1,1]^n.$

\begin{thm}[\hspace{1sp}\cite{deKlerkLaurent2018}]
\label{THM_Cheby_box}
Let $\mainset = [-1, 1]^n$ and consider the weight function $\widehat{\measureweight{\lambda}}$ from \eqref{EQ_boxweights}.
\begin{enumerate}
    \item[(i)] If $\lambda=-\frac{1}{2}$, then we have:
    \begin{equation}
        \label{EQ_THM_Cheby_box}
        \ubError{r}{\mainset, \widehat{\measureweight{\lambda}}}{f} = \bigO\bigg(\frac{1}{r^2}\bigg).
    \end{equation}
    \item[(ii)] If $n=1$ and $\mainfunc$ has a global minimizer on the boundary of $[-1, 1]$, then \eqref{EQ_THM_Cheby_box} holds for all $\lambda > -1$.
\end{enumerate}
\end{thm}
The key ingredients for claim (ii) above are: (a)  when the global minimizer is a boundary point of $[-1,1]$ then $\mainfunc$ has a linear upper estimator (recall Lemma~\ref{LEM_linear_upperestimator}), and (b) the convergence rate of 
\eqref{EQ_THM_Cheby_box} holds for any linear function and any $\lambda>-1$ (see \cite{deKlerkLaurent2018}).

In this section we show Theorem~\ref{THM_general_box} below, which extends the above result to all weight functions $\widehat{\measureweight{\lambda}}(x)$ with $\lambda \geq -\frac{1}{2}$.
Following the approach in \cite{deKlerkLaurent2018}, we proceed in two steps: first  we reduce to the univariate case, and then we deal with the univariate case.
Then the new situation to be dealt with is when $n=1$ and the minimizer lies in the interior of $[-1,1]$, which we can settle by getting back to the case $\lambda=-\frac{1}{2}$ through applying Proposition~\ref{PROP_differentmeasure}, the `local similarity' tool, with $\mainset = \biggerset = [-1,1]$.

\medskip
\noindent
{\bf Reduction to the univariate case.}
Let $a\in K$ be a global minimizer of $f$ in $K=[-1,1]^n$.  Following \cite{deKlerkLaurent2018} (recall Remark~\ref{REM_reducetoestimator} and
Lemma~\ref{LEM_quadraticupperestimator}),
we consider the upper estimator $\mainfunc(x) \leq_a g(x) := \mainfunc(a) + \langle \nabla \mainfunc(a) , x-a \rangle + \gamma_{f,K} ||x-a||^2$.
This $g$ is separable, i.e.,  we can write 
$g(x) = \sum_{i = 1}^n g_i(x_i)$, where each $g_i$ is quadratic univariate with $a_i$ as global minimizer over $[-1,1]$. 
Let $\Laspolye{r}^i$ be an optimum solution to the problem \eqref{EQ_fmin_measures_bounded} corresponding to the minimization of $g_i$ over $[-1,1]$ w.r.t. the weight function $w_{\lambda}(x_i)=(1-x_i^2)^{\lambda}$. 
If we set $q_r(x) = \prod_{i=1}^n q^i_r(x_i)$, then $q_r$ is a sum of squares with degree at most $nr$, such that $\int_\mainset q_r(x)\widehat{w_{\lambda}}(x)dx  = 1$. Hence we have
\begin{align*}
    \Lasupboundl{rn}{\mainset, \widehat{w_{\lambda}}} - f(a) &\leq \int_K f(x) q_r(x)  \widehat{w_{\lambda}} (x) dx -f(a) \\
        &\leq \int_{\mainset} g(x) \Laspolye{r}(x) \widehat{w_{\lambda}} (x) dx - g(a)\\
    &= \sum_{i=1}^n \bigg(\int_{-1}^1 g_i(x)\Laspolye{r}^i (x_i) w_{\lambda}(x_i)dx_i  - g_i(a_i) \bigg) \\
    &= \sum_{i=1}^n \big( (g_i)^{(r)}_{[-1, 1], w_{\lambda}} - g_i(a_i) \big)= \sum_{i=1}^n \ubError{r}{[-1, 1], w_{\lambda}}{g_i}.  
\end{align*}
As a consequence, we need only to consider the case of a quadratic univariate polynomial $\mainfunc$ on $\mainset = [-1, 1]$. 
We  distinguish two cases, depending whether the global minimizer  lies on the boundary or in the interior of $K$.
The case when the global minimizer lies on the boundary of $[-1,1]$ is  settled by Theorem~\ref{THM_Cheby_box}(ii) above, so
we next assume the global minimizer lies in the interior of $[-1,1]$.

\medskip
\noindent
\textbf{Case of a global minimizer  in the interior of $K=[-1,1]$.}
To deal with this case we make use of Proposition~\ref{PROP_differentmeasure} with $K=\widehat K=[-1,1]$, weight function $w(x) := w_{\lambda}(x)$ on $\mainset$, and weight function $\widehat{w}(x) := w_{-1/2}(x)$ on $\biggerset$.
We check that the conditions of the proposition are met. As $\biggerset = \mainset$, clearly we have $\mainset \subseteq_a \biggerset$.
Further, for any $\lambda \geq -\frac{1}{2}$, we have
\begin{equation}
\measureweight{\lambda}(x) = (1-x^2)^{\lambda} \leq (1-x^2)^{-\frac{1}{2}} = \measureweight{-1/2}(x)
\end{equation}	
for all $x \in (-1, 1) = \interior{\mainset}$. As $a \in \interior{\mainset}$, we also have $\measureweight{\lambda} \measureleq{a} \measureweight{-1/2}$ (see \mbox{Remark~\ref{REM_weightdominance}(i)}).
Hence we may apply Proposition~\ref{PROP_differentmeasure} to find that there exists a polynomial upper estimator $g$ of $\mainfunc$ on $[-1, 1]$, exact at $a$, and having
\begin{equation}
	   \ubError{r}{\mainset, \measureweight{}}{g} \leq \frac{2}{m_a}  \ubError{r}{\biggerset, \widehat{\measureweight{}}}{g}
\end{equation}
for all $r \in \N$ large enough.
Now, (the univariate case of) Theorem~\ref{THM_Cheby_box}(i) allows us to claim $\ubError{r}{\biggerset, \widehat{\measureweight{}}}{g}=
\bigO(1/r^2)$, so that  we obtain:
\begin{equation}
	\ubError{r}{\mainset, \measureweight{\lambda}}{\mainfunc} 
	\le \ubError{r}{K,w_\lambda}{g}
	= 
	\bigO_a \big(\ubError{r}{\biggerset, \widehat{\measureweight{}}}{g}\big) =
	\bigO_a({1 / r^2}).
\end{equation}

In summary, in view of the above, 
we have shown  the following extension of Theorem~\ref{THM_Cheby_box}.

\begin{thm}
\label{THM_general_box}
Let $\mainset = [-1, 1]^n$ and  $\lambda \geq -\frac{1}{2}$. Let $a$ be a global minimizer of $\mainfunc$ on $\mainset$. Then we have
\begin{equation}
    \ubError{r}{\mainset, \widehat{\measureweight{\lambda}}}{\mainfunc} = \bigO_a\bigg(\frac{1}{r^2}\bigg).
\end{equation}
\end{thm}
The constant $m_a$ involved in the proof of Theorem~\ref{THM_general_box} depends on the global minimizer $a$ of $\mainfunc$ on $[-1,1]$. It is introduced by the application of Proposition~\ref{PROP_differentmeasure} to cover the case where $a$ lies in the interior of $[-1,1]$. When $\lambda = 0$ (i.e., when $w = w_0 = 1$ corresponds to the Lebesgue measure), one can replace $m_a$ by a uniform constant $m > 0$,  as we now explain.

Consider $\biggerset := [-2, 2] \supseteq [-1, 1] = K$, equipped with the scaled Chebyshev weight $\widehat{\measureweight{}}(x) := \measureweight{-1/2}(x / 2)=(1-x^2/4)^{-1/2}$. Of course, Theorem~\ref{THM_Cheby_box} applies to this choice of $\widehat{\mainset}, \widehat{\measureweight{}}$ as well. 
Further, we still have $\widehat{\measureweight{}}(x) \geq \measureweight{}(x) = \measureweight{0}(x) = 1$ for all $x \in [-1,1]$.
However, we now have a \emph{uniform} upper bound $\widehat{\measureweight{}}(x) \leq \widehat{\measureweight{}}(1)$ for $\widehat{\measureweight{}}$ on $\mainset$, which means that $\widehat{\measureweight{}} \measureleq{} \measureweight{}$ on $\mainset$ (see Remark~\ref{REM_weightdominance}(ii)). Indeed, we have
\begin{equation}
\widehat w(x)/\widehat w(1)\le 1 = w_0(x) = \measureweight{}(x) \quad \text{ for all } x\in [-1,1].
\end{equation}
We may thus apply Corollary~\ref{COR_uniformdifferentmeasure} (instead of Proposition~\ref{PROP_differentmeasure}) to obtain the following.
\begin{cor}
\label{COR_general_box}
If $\mainset = [-1, 1]^n$ is equipped with the Lebesgue measure then
\begin{equation}
    \ubError{r}{\mainset}{\mainfunc} = \bigO\bigg(\frac{1}{r^2}\bigg).
\end{equation}
\end{cor}

\subsection{The unit ball}
We now consider optimization over the unit ball $K=B^n\subseteq \R^n$ ($n\ge 2$); we restrict to reference measures on $B^n$ with 
weight function  of the form 
\begin{equation}
    \label{EQ_ballweights}
    \measureweight{\lambda}(x)  = (1-||x||^2)^\lambda,
\end{equation}
where $\lambda > -1$.  For further reference we recall (see e.g. \cite[\S 6.3.2]{Xu2017}) or \cite[\S 11]{DaiXu2013}) that
\begin{equation}\label{EQ:Cnlambda}
C_{n,\lambda}:=\int_{B^n}w_{\lambda}(x)dx = {\pi^{n\over 2} \Gamma(\lambda+1)\over \Gamma\big(\lambda+1+{n\over 2}\big)}.
\end{equation} 
For the case $\lambda \ge 0$, we can analyse the bounds and show the following result.

\begin{thm}
\label{THM_general_ball}
Let $\mainset = \unitball$ be the unit ball. Let $a$ be a global minimizer of $\mainfunc$ on $\mainset$. Consider the weight function $\measureweight{\lambda}$ from \eqref{EQ_ballweights} on $\mainset$.
\begin{enumerate}
    \item[(i)] If $\lambda = 0$, we have
    \begin{equation}
    \ubError{r}{\mainset, \measureweight{\lambda}}{\mainfunc} = \bigO\bigg(\frac{1}{r^2}\bigg).
    \end{equation}
    \item[(ii)] If $\lambda > 0$, we have
    \begin{equation}
    \ubError{r}{\mainset, \measureweight{\lambda}}{\mainfunc} = \bigO_a\bigg(\frac{1}{r^2}\bigg).
    \end{equation}
\end{enumerate}
\end{thm}
For the proof, we distinguish the two cases when $a$ lies in the interior of $K$ or on its boundary.

\medskip
\noindent
\textbf{Case of a global minimizer in the interior of $\mainset$.}
Our strategy is to reduce this  to the case of the hypercube with the help of Proposition~\ref{PROP_differentmeasure}. Set $\biggerset := [-1, 1]^n \supseteq \unitball = \mainset$. As $a \in \interior{\mainset}$, we have $\mainset \subseteq_a \biggerset$. 
Consider the weight function $w(x) := w_{\lambda}(x)=(1-\|x\|^2)^{\lambda}$ on $\mainset$, and $\widehat{w}(x) := 1$ on the hypercube $\biggerset$. Since $\lambda \ge 0$, we have $w_{\lambda}(x)\le 1\le \widehat{w}(x)$ for all $x \in K$. Furthermore, as $a \in \interior{\mainset}$, we also have $\widehat{w} \measureleq{a} w$. Hence we may apply Proposition~\ref{PROP_differentmeasure} to find a polynomial upper estimator $g$ of $\mainfunc$ on $\biggerset$, exact at $a$, satisfying
 \begin{equation}
	\ubError{r}{\mainset, \measureweight{}}{g} \leq \frac{2}{m_a} \ubError{r}{\biggerset, \widehat{\measureweight{}}}{g}
\end{equation}
for all $r \in \N$ large enough. Here $m_a > 0$ is the constant from \eqref{EQ_measurelocaldominance}. Now, Theorem~\ref{THM_general_box} allows us to claim $\ubError{r}{\biggerset, \widehat{\measureweight{}}}{g} = \bigO_a(1/r^2)$. Hence we obtain:
\begin{equation}
	\ubError{r}{\mainset, \measureweight{}}{\mainfunc} \le
	\ubError{r}{K,w}{g} 
	=
	\bigO_a(\ubError{r}{\biggerset, \widehat{\measureweight{}}}{g}) =
	\bigO_a({1 / r^2}).
\end{equation}

As in the previous section, it is possible to replace the constant $m_a$ by a uniform constant $m > 0$ in the case that $\lambda = 0$, i.e., in the case that we have the Lebesgue measure on $\mainset$. Indeed, in this case we have $\widehat{w} = w \ (=w_0=1)$, and so in particular $\widehat{w} \measureleq{} w$. 
We may thus invoke Corollary~\ref{COR_uniformdifferentmeasure} (instead of Proposition~\ref{PROP_differentmeasure}) to obtain
\begin{equation}
	\ubError{r}{\mainset, \measureweight{}}{g} \leq 2 \ubError{r}{\biggerset, \widehat{\measureweight{}}}{g}
\end{equation}
and so
\begin{equation}
	\ubError{r}{\mainset, \measureweight{}}{\mainfunc} = 
	\bigO(\ubError{r}{\biggerset, \widehat{\measureweight{}}}{g}) =
	\bigO({1 / r^2}).
\end{equation} 
Note that in this case, we do not actually make use of the fact that $\mainset = \unitball$. Rather, we only need that $a$ lies in the interior of $\mainset$ and that $\mainset \subseteq [-1, 1]^n$. As we may freely apply affine transformations to $\mainset$ (by Lemma~\ref{LEM_lineartransformation}), the latter is no true restriction. We have thus shown the following result.
\begin{thm}
\label{THM_interiorglobalminimizer}
Let $\mainset \subseteq \R^n$ be a compact set, with non-empty interior, equipped with the Lebesgue measure. Assume that $\mainfunc$ has a global minimizer $a$ on $\mainset$ with $a \in \interior{\mainset}$. Then we have
\begin{equation}
    \ubError{r}{\mainset}{\mainfunc} = \bigO\bigg(\frac{1}{r^2}\bigg).
\end{equation}
\end{thm}

\medskip
\noindent
{\bf Case of a global minimizer on the boundary of $\mainset$.} 
Our strategy is now to reduce to the univariate case of the interval $[-1,1]$.
For this, we use Lemma~\ref{LEM_linear_upperestimator}, which claims that $f$ has a linear upper estimator $g$ on $K$, exact at $a$.
Up to applying an orthogonal transformation (and scaling) we may assume that $g$ is of the form $g(x)=x_1$. It therefore suffices now to analyze the behaviour of the bounds for the function $x_1$ minimized on the ball $\unitball$. Note that when minimizing  $x_1$ on $B^n$ or on the interval $[-1,1]$ the minimum is attained at the boundary in both cases.
The following technical lemma will be useful for reducing to the case of  the interval $[-1,1]$.

\begin{lem}
\label{LEM_balltounivariate}
Let $h$ be a univariate polynomial and let $\lambda>-1$. Then we have
\begin{equation}
\int_{\unitball} h(x_1)  \measureweight{\lambda}(x) dx = C_{n-1,\lambda} \int_{-1}^{1} h(x_1) w_{\lambda + \frac{n-1}{2}}(x_1) dx_1,
\end{equation}
where  $C_{n-1,\lambda}$ is given in \eqref{EQ:Cnlambda}.
\end{lem}

\begin{proof}
Change variables and set $u_j= {x_j\over \sqrt{1-x_1^2}}$ for $2\le j\le d$. Then we have
 \begin{equation}w_{\lambda}(x)=(1-x_1^2 -x_2^2+\ldots -x_n^2)^{\lambda}= (1-x_1^2)^{\lambda}(1-u_2^2-\ldots -u_n^2)^{\lambda}\end{equation}
 and
 $dx_2\cdots dx_n= (1-x_1^2)^{n-1\over 2} du_2\cdots du_n.$
 Putting things together we obtain 
 the desired result.
\qedMP\end{proof}

Let $q_r(x_1)$ be an optimal sum-of-squares density with degree at most $2r$ for the problem of minimizing $x_1$ over the interval $[-1,1]$, equipped with the weight function $w(x):=w_{\lambda +{n-1\over 2}}(x)$. Then,  its scaling $C_{n-1,\lambda}^{-1} q_r(x_1)$  provides a  feasible solution for the problem of minimizing $g(x)=x_1$ over the ball $K=B^n$. Indeed, using Lemma~\ref{LEM_balltounivariate},  we have 
$\int_{B^n} C_{n-1,\lambda}^{-1}q_r(x_1)w_{\lambda}(x)dx= \int_{-1}^1q_r(x_1) w(x)dx_1=1$, and so
\begin{equation}
g^{(r)}_{K,w_{\lambda}} \le \int_{B^n} x_1 C_{n-1,\lambda }^{-1} q_r(x_1) w_{\lambda}(x)dx = \int_{-1}^1 x_1 q_r(x_1) w(x_1)dx_1.
\end{equation}
The proof is now concluded by applying Theorem~\ref{THM_Cheby_box}(ii).

\subsection{Ball-like convex bodies}\label{SEC:round}

Here we show a convergence rate of $\ubError{r}{\mainset}{\mainfunc}$ in $\bigO(1/r^2)$ for a special class of smooth convex bodies $\mainset$ with respect to the Lebesgue measure. The basis for this result is a reduction to the case of the unit ball.

We say $\mainset$ has an \emph{inscribed tangent ball} (of radius $\epsilon$) at $x \in \partial K$ if there exists  $\epsilon > 0$ and a closed ball $B_{insc}$ of radius $\epsilon$ such that $x \in \partial B_{insc}$ and $B_{insc} \subseteq \mainset$. Similarly, we say $\mainset$ has a \emph{circumscribed tangent ball} (of radius $\epsilon$) at $x \in \partial K$ if there exists  $\epsilon > 0$ and a closed ball $B_{circ}$ of radius $\epsilon$ such that $x \in \partial B_{circ}$ and $\mainset \subseteq B_{circ}$. 

\begin{defn}\label{DEF:round}
We say that  a (smooth) convex body $\mainset$ is  \emph{ball-like} if there exist (uniform) $\epsilon_{insc}, \epsilon_{circ} > 0$ such that $\mainset$ has inscribed and circumscribed tangent balls of radii $\epsilon_{insc}, \epsilon_{circ}$, respectively, at all points $x \in \partial K$.
\end{defn}

\begin{thm}
\label{THM_squeeze}
Assume that $\mainset$ is a (smooth) ball-like convex body, equipped with the Lebesgue measure. Then we have
\begin{equation}
\ubError{r}{\mainset}{\mainfunc} = \bigO\bigg(\frac{1}{r^2}\bigg).
\end{equation}
\end{thm}
\begin{proof}
Let $a \in \mainset$ be a global minimizer of $\mainfunc$ on $\mainset$.
We again distinguish  two cases depending on whether $a$ lies in the interior of $\mainset$ or on its boundary.

\medskip
\noindent
\textbf{Case of a global minimizer in the interior of $\mainset$.} This case is covered directly by Theorem~\ref{THM_interiorglobalminimizer}.

\medskip
\noindent
\textbf{Case of a global minimizer on the boundary of $\mainset$.}
By applying a suitable affine transformation, we can arrange that the following holds:  $f(a)=0$, $a = 0$, $e_1$ is an inward normal of $\mainset$ at $a$, and  the radius of the circumscribed tangent ball $B_{circ}$ at $a$ is equal to 1, i.e.,  $B_{circ} = \ball{1}{e_1}$. See Figure~\ref{FIG_squeeze} for an illustration.
Now, as $a$ is a global minimizer of $\mainfunc$ on $\mainset$, we have $\nabla \mainfunc(a) \in \normalcone{\mainset}{a}$ by Lemma~\ref{LEM_first_order_condition}. But $\normalcone{\mainset}{a} = \normalcone{B_{circ}}{a}$, and so $\nabla \mainfunc(a) \in \normalcone{B_{circ}}{a}$. As noted in Remark~\ref{REM_gradientinnormalcone}, we may thus use Lemma~\ref{LEM_linear_upperestimator} to find that $f(x) \leq_a c \langle e_1, x \rangle = cx_1$ on $B_{circ}$ for some constant $c > 0$. In light of \mbox{Remark~\ref{REM_reducetoestimator}(i)}, and after scaling, it therefore suffices to analyse the function $\mainfunc(x) = x_1$.

Again, we will use a reduction to the univariate case, now on the interval $[0, 2]$.
For any $r \in \N$, let $\Laspolye{r} \in \Sigma_r$ be an optimum sum-of-squares density of degree $2r$ for the minimization of $x_1$ on $[0, 2]$ with respect to the weight function 
\[ w'(x_1) := \measureweight{\frac{n-1}{2}}(x_1 - 1) = [1 - (x_1 - 1)^2]^{\frac{n-1}{2}} = [2x_1 - x_1^2]^{\frac{n-1}{2}}.
\]
That is, $q_r \in \Sigma_r$ satisfies
\begin{equation}\label{EQ:x1}
	\int_0^2 x_1 \Laspolye{r} (x_1) w'(x_1) dx_1 = \bigO(1/r^2) \quad \text{ and }\quad \int_0^2 \Laspolye{r} (x_1) w'(x_1) dx_1 = 1,
\end{equation}
where the first equality relies on Theorem~\ref{THM_Cheby_box}(ii). 
As $x \mapsto \Laspolye{r}(x_1) / (\int_{\mainset} \Laspolye{r}(x_1) dx)$ is a sum-of-squares density on $\mainset$ with respect to the Lebesgue measure, we have
\begin{equation}
\label{EQ_squuezefrac}
\ubError{r}{\mainset}{\mainfunc} \leq \frac{\int_{\mainset} x_1 \Laspolye{r}(x_1) dx}{\int_{\mainset} \Laspolye{r}(x_1) dx}. 
\end{equation}
We will now show that, on the one hand, the numerator $\int_{\mainset} x_1 \Laspolye{r}(x_1)  dx$ in \eqref{EQ_squuezefrac} has an upper bound in 
$\bigO(1/r^2)$ and that, on the other hand,  the denominator $\int_{\mainset} \Laspolye{r}(x_1) dx $ in \eqref{EQ_squuezefrac} is lower bounded by an absolute constant that does not depend on $r$. Putting these two bounds together then yields $\ubError{r}{\mainset}{\mainfunc} = \bigO(1/r^2)$, as desired.

\medskip 
\noindent
\textbf{The upper bound.} We make use of the fact that $\mainset \subseteq B_{circ}$ to compute:
\begin{align}
	\int_{\mainset} x_1 \Laspolye{r}(x_1) dx
	&\leq \int_{B_{circ}} x_1  \Laspolye{r}(x_1) dx \\
	&= \int_{\unitball} (y_1+1) \Laspolye{r}(y_1+1)  dy & [y = x - e_1 ] \\ 
	&= C_{n-1, 0} \int_{-1}^1 (y_1+1) \Laspolye{r}(y_1+1) w_{\frac{n-1}{2}}(y_1) dy_1 & [\text{by Lemma } \ref{LEM_balltounivariate}] \\
	&= C_{n-1, 0} \int_{0}^2 z \Laspolye{r}(z)  w'(z) dz  &[z = y_1 + 1 ] \\ 
	&= \bigO(1/r^2). &[\text{by } \eqref{EQ:x1}]
\end{align}

\medskip 
\noindent
\textbf{The lower bound.}
Here, we consider an inscribed tangent ball $B_{insc}$ of $\mainset$ at $a = 0$. Say $B_{insc} = \ball{\rho}{\rho e_1}$ for some $\rho > 0$. See again Figure~\ref{FIG_squeeze}. We may then compute:
\begin{align*}
     \int_{\mainset} \Laspolye{r}(x_1) dx &\geq \int_{B_{insc}} \Laspolye{r}(x_1) dx \\
     &= \int_{\unitball} \Laspolye{r}\big(\rho (y_1 + 1)\big) \rho^n dy 	& [y = \frac{x - \rho e_1}{\rho}]  \\ 
     &= \rho^n C_{n-1, 0} \int_{-1}^1 \Laspolye{r}\big(\rho (y_1 + 1)\big) w_{\frac{n-1}{2}}(y_1)dy_1 & [\text{by Lemma } \ref{LEM_balltounivariate}] \\
     &= \rho^{n-1} C_{n-1, 0} \int_{0}^{2 \rho} \Laspolye{r}(z) w_{\frac{n-1}{2}}(z / \rho - 1)  dz 
     & [z = \rho (y_1 + 1) ]\\ 
     & \ge \rho^{n-1} C_{n-1, 0}\int_0^\rho q_r(z) w'(z) {w_{{n-1\over 2}}(z/\rho -1) \over w_{{n-1\over 2}}(z-1)} dz
     & [w'(z)=w_{{n-1\over 2}}(z-1)]\\   
     & \ge \bigg({\rho\over 2-\rho}\bigg)^{n-1\over 2} C_{n-1,0}  \int_0^\rho q_r(z) w'(z)dz,\\
\end{align*}
where the last inequality follows using the fact that ${1-(z/\rho-1)^2 \over 1-(z-1)^2}\ge {1\over \rho(2-\rho)}$ for $z\in [0,\rho]$.
It remains to show that 
$$ \int_{0}^{\rho} \Laspolye{r}(z) w'(z)dz \geq {1\over 2}\quad \text{ for all } r \text{ large enough.}$$
The argument is similar to the one used for the proof of Lemma~\ref{LEM_ballweight}.
By \eqref{EQ:x1}, there is a constant $C>0$ such that  $\int_0^2 z q_r(z) w'(z)dz \le {C\over r^2}$ for all $r\in \N$. So we have
\[
	{C\over r^2}\ge \int_\rho^2 zq_r(z)w'(z)dz\ge \rho\int_\rho^2q_r(z)w'(z)dz = \rho\bigg(1-\int_0^\rho q_r(z)w'(z)dz\bigg),
\]
which implies 
$\int_0^\rho q_r(z)w'(z)dz \ge 1-{C\over \rho r^2}\ge {1\over 2}$ for $r$ large enough.

This concludes the proof of Theorem~\ref{THM_squeeze}.
\qedMP\end{proof}

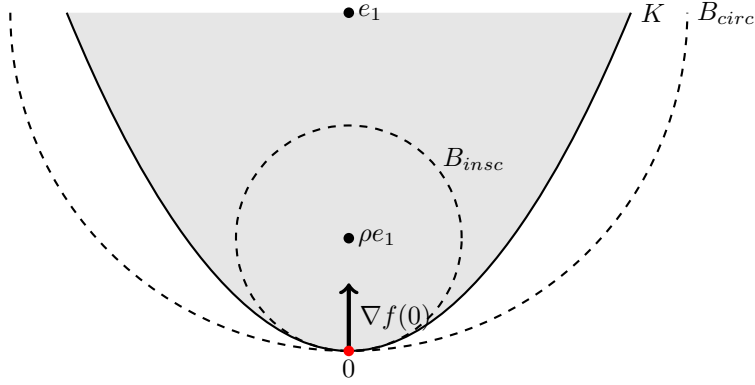
\begin{figure}
    \centering
    \begin{tikzpicture}[x=1.5cm,y=1.5cm]
    \draw[->, ultra thick] (0, 0) -- node[right] {$\nabla f(0)$} (0, 0.6);

    \draw[below] (0,0) node {$0$};
    
    \draw[thick, fill, fill opacity = 0.1] (-2.5,3) parabola bend (0,0) (2.5,3) node[right, opacity=1, black] {$\mainset$};
    
    \draw[thick, dashed] (0,1) circle (1);
    \fill[] (0, 1) circle (2pt);
    \draw[right] (0.75,1.7) node {$B_{insc}$};
    \draw[right] (0,1) node {$\rho e_1$};
    
    \draw[thick, dashed] (0,3) +(-180:3) arc (-180:0:3) node[right, black] {$B_{circ}$};
    \fill[] (0, 3) circle (2pt);
    \draw[right] (0,3) node {$e_1$};

    \fill[red] (0, 0) circle (2pt);

\end{tikzpicture}
    \caption{An overview of the situation in the second case of the proof of Theorem~\ref{THM_squeeze}.}
    \label{FIG_squeeze}
\end{figure}

\medskip
\noindent
\textbf{Classification of ball-like sets.} 
With Theorem~\ref{THM_squeeze} in mind, it is interesting to understand under which conditions a convex body $\mainset$ is ball-like. 
Under the assumption that $\mainset$ has a $C^2$-boundary, the well-known Rolling Ball Theorem (cf., e.g., \cite{Koutroufiotis1972}) guarantees the existence of inscribed tangent balls.

\begin{thm}[Rolling Ball Theorem]
\label{THM_rollingball}
Let $\mainset \subseteq \R^n$ be a convex body with $C^2$- boundary. Then there exists  $\epsilon_{insc} > 0$ such that $\mainset$ has an inscribed tangent ball of radius $\epsilon_{insc}$ for each $x \in \partial \mainset$.
\end{thm}
Classifying the existence of circumscribed tangent balls is somewhat more involved. Certainly, we should assume that $\mainset$ is \emph{strictly} convex, which means that  its boundary should not contain any line segments. 
This assumption, however, is not sufficient.
 Instead we need the following  stronger notion of {\em 2-strict convexity}  introduced in \cite{DallaHatziafratis2006}.

\begin{defn}
Let $\mainset \subseteq \R^n$ be a convex body with $C^2$-boundary and let $\Psi \in C^2(\R^{n}, \R)$ such that $\mainset = \Psi^{-1}((-\infty, 0])$ and $\partial \mainset = \Psi^{-1}(0)$. Assume $\nabla \Psi(a)\ne 0$ for all $a\in \partial K$.
The set $K$ is said to be {\em $2$-strictly convex} if the following holds:
\[
	x^T\nabla ^2 \Psi(a) x >0 \quad \text{ for all } x\in T_a K \setminus\{0\} \text{ and } a\in \partial K.
\]
In other words, the Hessian of $\Psi$ at any boundary point  should be positive definite, when restricted to the tangent space.
\end{defn}

\begin{exmp}
Consider the unit ball for the $\ell_4$-norm:
\begin{equation} 
  K=  \{ (x_1, x_2) : \Psi(x_1,x_2):=x_1^4 + x_2^4 \leq 1\} \subseteq \R^2.
\end{equation}
Then, $K$ is strictly convex, but 
 not 2-strictly convex.  Indeed, at any of the  points $a= (0, \pm 1)$ and $(\pm 1,0)$, the Hessian of $\Psi$ is not positive definite on the tangent space. For instance, for $a=(0,-1)$, we have $\nabla\Psi(a)= (0,-4)$ and 
$x^T\Psi^2(a)x=12x_2^2$, which vanishes at $x=(1,0)\in T_aK$.
In fact, one can verify that $K$ does not have a circumscribed tangent ball at any of the  points $ (0, \pm 1)$, $(\pm 1,0)$.
\end{exmp}

It is shown in \cite{DallaHatziafratis2006} that the set of 2-strictly convex bodies lies dense in the set of all convex bodies.
For $\mainset$ with $C^2$-boundary, it turns out that $2$-strict convexity is equivalent to the existence of circumscribed tangent balls at all boundary points.

\begin{thm}[\hspace{1sp}{\cite[Corollary 3.3]{DallaSamiou2007}}]
\label{THM_outerball}
Let $\mainset$ be a convex body with $C^2$-boundary. Then $\mainset$ is $2$-strictly convex if and only if there exists $\epsilon_{circ} > 0$ such that $\mainset$ has a circumscribed tangent ball of radius $\epsilon_{circ}$ at all boundary points $a \in \partial \mainset$.
\end{thm}
Combining \mbox{Theorems \ref{THM_rollingball} and \ref{THM_outerball}} then gives a full classification of the ball-like convex bodies $\mainset$ with $C^2$-boundary.

\begin{cor}
\label{COR_round_2strictlyconvex}
Let $\mainset \subseteq \R^n$ be a convex body with $C^2$-boundary. Then $\mainset$ is ball-like if and only if it is $2$-strictly convex.
\end{cor}

\noindent
\textbf{A convex body without inscribed tangent balls.} We now give an example of a convex body $\mainset$ which does not have inscribed tangent balls, going back to de Rham \cite{deRham1947}. 
The idea is to construct a curve by starting with a polygon, and then successively `cutting corners'. Let $C_0$ be the polygon in $\R^2$ with vertices $(-1, -1), (1, -1), (1, 1)$ and $(-1, 1)$, i.e., a square. 
For $k \geq 1$, we obtain $C_k$ by subdividing each edge of $C_{k-1}$ into three equal parts and taking the convex hull of the resulting subdivision points (see Figure~\ref{FIG_deRhamCurve}). We then let $C$ be the limiting curve obtained by letting $k$ tend to $\infty$.
Then, $C$ is a continuously differentiable, convex curve (see \cite{deBoor1986} for details). It is not, however, $C^2$ everywhere.
We indicate below some point where no inscribed tangent ball exists for the convex body with boundary $C$.

Consider the point $m = (0,-1) \in C$, which is an element of $C_k$ for all $k$. Fix $k\ge 1$.
If we walk anti-clockwise along $C_k$ starting at $m$, the first corner point encountered is $s_k = (1/3^k, -1)$, the slope of the edge starting at $s_k$ is $l_k=1/k$ and its end point is 
\begin{equation}
e_k = \big((2k + 1) / 3^k, 2 / 3^k - 1\big). 
\end{equation}
Now suppose that there exists an inscribed tangent ball $B_\epsilon (c)$ at the point $m$. Then, $\epsilon>0$, $c=(0,\epsilon -1)$ and any point $(x,y)\in C$ lies outside of the ball $B_\epsilon(c)$, so that 
$$x^2+(y+1)^2-2\epsilon(y+1)\ge 0 \quad \text{ for all } (x,y)\in C.$$
As $C$ is contained in the polygonal region delimited by any $C_k$, also $e_k\not\in B_\epsilon(c)$ and thus 
$ \big({2k+1\over 3^k}\big) ^2 + \big({2\over 3^k}\big)^2 -{4\epsilon \over 3^k}\ge 0$. Letting $k\to\infty$, we get $\epsilon=0$, a contradiction.

\begin{figure}[]
    \centering
    \scalebox{1}{ \begin{tikzpicture}[x=1.3cm, y=1.3cm]
\draw[thick] (-1.0,-1.0) -- (1.0,-1.0) -- (1.0,1.0) -- (-1.0,1.0) -- cycle;
\end{tikzpicture} }
    \scalebox{1}{ \begin{tikzpicture}[x=1.3cm, y=1.3cm]
\draw[thick] 
(-0.333,-1.0) -- (0.333,-1.0) -- (1.0,-0.333) -- (1.0,0.333) -- (0.333,1.0) -- (-0.333,1.0) -- (-1.0,0.333) -- (-1.0,-0.333) -- cycle;

%

\end{tikzpicture} }
    \scalebox{1}{ \begin{tikzpicture}[x=1.3cm, y=1.3cm]
\draw[thick] (-0.111,-1.0) -- (0.111,-1.0) -- (0.556,-0.778) -- (0.778,-0.556) -- (1.0,-0.111) -- (1.0,0.111) -- (0.778,0.556) -- (0.556,0.778) -- (0.111,1.0) -- (-0.111,1.0) -- (-0.556,0.778) -- (-0.778,0.556) -- (-1.0,0.111) -- (-1.0,-0.111) -- (-0.778,-0.556) -- (-0.556,-0.778) -- cycle;
\end{tikzpicture} }
    \scalebox{1}{ \input{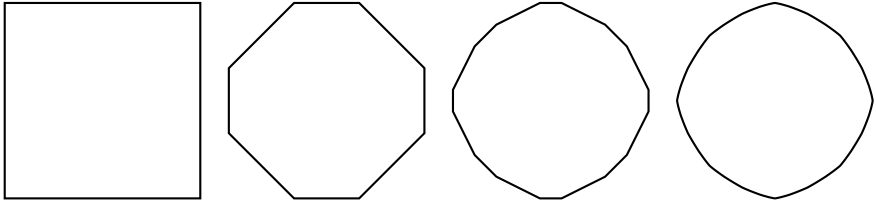} }
    \caption{From left to right: the curve $C_k$ for $k = 0,1,2,8$.}
    \label{FIG_deRhamCurve}
\end{figure}

\subsection{The simplex}
We now consider a full-dimensional simplex $\simplex{n} := \conv(\{v_0, v_1, v_2, \dots, v_n\}) \subseteq \R^{n}$, equipped with the Lebesgue measure. We show the following.
\begin{thm}
\label{THM_simplex}
Let $\mainset = \simplex{n}$ be a simplex, equipped with the Lebesgue measure. Then 
\begin{equation}
\ubError{r}{\simplex{n}}{\mainfunc} = \bigO\bigg(\frac{1}{r^2}\bigg).
\end{equation}
\end{thm}
\begin{proof}
Let $a \in \simplex{n}$ be a global minimizer of $\mainfunc$ on $\simplex{n}$.
The idea is to apply an affine transformation $\phi$ to $\simplex{n}$ whose image $\phi(\simplex{n})$ is locally similar to $[0,1]^n$ at the global minimizer $\phi(a)$ of $g := \mainfunc \circ \phi^{-1}$, after which we may `transport' the $\bigO(1/r^2)$ rate from the hypercube to the simplex.
 
Let $F := \conv(v_1, v_2, \dots, v_n)$ be the facet of $\simplex{n}$ which does not contain $v_0$. By reindexing, we may assume w.l.o.g. that $a \not \in F$.
Consider the map $\phi$ determined by $\phi(v_0) = 0$ and $\phi(v_i) = e_i$ for all $i \in [n]$, where $e_i$ is the $i$-th standard basis vector of $\R^n$. See Figure~\ref{FIG_simplex}. Clearly, $\phi$ is nonsingular, and $\phi(\simplex{n}) \subseteq [0,1]^n$.

\begin{lem}
We have $\phi(\simplex{n}) \subseteq_{\phi(x)} [0, 1]^n$ for all $x \in \simplex{n} \setminus F$.
\end{lem}
\begin{proof}
By definition of $F$, we have
\begin{equation}
\simplex{n} \setminus F = \bigg\{\sum_{i=0}^n \lambda_i v_i : \sum_{i=1}^n \lambda_i < 1, \lambda \geq 0 \bigg\},
\end{equation}
and so
\begin{equation}
\phi(\simplex{n} \setminus F) = \big\{y \in [0,1]^n : \sum_{i=1}^n y_i < 1\big\},
\end{equation}
which is an open subset of $[0, 1]^n$.
But this means that for each $y = \phi(x) \in \phi(\simplex{n} \setminus F)$ there exists  $\epsilon > 0$ such that
\begin{equation}
 \ball{\epsilon}{y} \cap [0,1]^n \subseteq \ball{\epsilon}{y} \cap \phi(\simplex{n} \setminus F),
\end{equation}
which concludes the proof of the lemma. \qedMP\end{proof}
The above lemma tells us in particular that $\phi(\simplex{n}) \subseteq_{\phi(a)} [0,1]^n$. We now apply Corollary~\ref{COR_uniformdifferentmeasure} with $\mainset = \phi(\simplex{n})$, $\biggerset = [0,1]^n$ and weight functions $w = \widehat{w} = 1$ on $\mainset, \biggerset$, respectively. This yields a polynomial upper estimator $h$ of $g$ on $[0,1]^n$ having 
\begin{equation}
\ubError{r}{\phi(\simplex{n})}{g} \leq 2 \ubError{r}{[0,1]^n}{h} = \bigO(1/r^2),
\end{equation}
for $r \in \N$ large enough, using Theorem~\ref{THM_general_box} for the right most equality.
It remains to apply Lemma~\ref{LEM_lineartransformation} to obtain:
\begin{equation}
    \ubError{r}{\simplex{n}}{\mainfunc} = \ubError{r}{\phi(\simplex{n})}{g} = \bigO(1/r^2),
\end{equation}
which concludes the proof of Theorem~\ref{THM_simplex}.
\qedMP\end{proof}

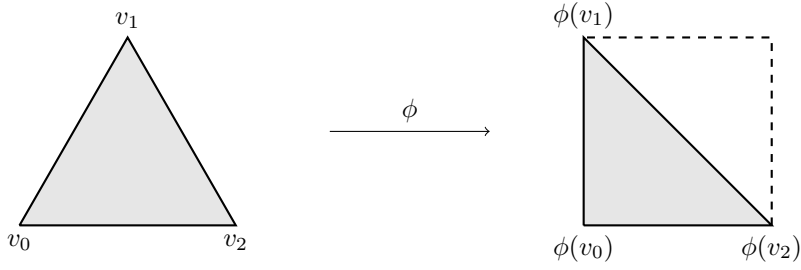
\begin{figure}
    \centering
    \begin{tikzpicture}[x=2.5cm,y=2.5cm]
    \draw[dashed, thick] (1, 0) -- (1, 1) -- (0,1) ;
    
    \draw[thick] (0,0) -- (1, 0)-- (0,1) -- (0,0);
    \draw[opacity=0.1, fill] (0,0) -- (1, 0) -- (0,1) -- (0,0);
    
    \draw[->] (-1.35 , 0.5) -- node[above] {$\phi$} (-0.5 , 0.5); 
    
    \draw (0,0) node[below] {$\phi(v_0)$};
    \draw (0,1) node[above] {$\phi(v_1)$} ;
    \draw (1,0) node[below] {$\phi(v_2)$} ;
    

    \begin{scope}[shift={(-3,0)}]
        \draw[thick] (0,0) -- (0.575, 1) -- (1.15,0) -- (0,0);
        \draw[opacity=0.1, fill] (0,0) -- (0.575, 1) -- (1.15,0) -- (0,0);

        \draw (0,0) node[below] {$v_0$};
        \draw (0.575, 1) node[above] {$v_1$} ;
        \draw (1.15,0) node[below] {$v_2$} ;

    \end{scope}



\end{tikzpicture}
    \caption{The map $\phi$ from the proof of Theorem~\ref{THM_simplex} for $n=2$}
    \label{FIG_simplex}
\end{figure}

\section{General sets}
    \label{SEC_GeneralSets}
    In this section we analyze the error $\ubError{r}{}{\mainfunc}$ for a general compact set $K$ equipped with the Lebesgue measure.
We will show  the following two results:
when $K$ satisfies a mild assumption (Assumption~\ref{ASSU_one}) we prove a convergence rate in $\bigO(\log r/r)$ (Theorem~\ref{THM_bound_interior_cone}),  which improves on the previous rate in $\bigO(1/\sqrt r)$ from \cite{deKlerk2017}, and when $K$ is a convex body we prove a convergence rate in $\bigO((\log r/ r)^2)$
(Theorem~\ref{THM_bound_convex_body}),  improving the previous rate in $\bigO(1/r)$ from \cite{deKlerkLaurent2017}.
As a byproduct of our analysis, we can show the stronger bound $\bigO((\log r / r)^\beta)$ when all partial derivatives of $\mainfunc$ of order at most $\beta-1$ vanish at a global minimizer (see Theorem~\ref{THM_fasterconvergence}).
We begin with introducing Assumption~\ref{ASSU_one}.

\begin{assu}
\label{ASSU_one}
There exist constants $\assuoneepsilon, \assuoneeta > 0$ such that 
\begin{equation}
    \label{EQ_interior_cone}
    \vol\big(\ball{\delta}{x} \cap K\big) \geq \assuoneeta \vol (B^n_\delta(x)) = \delta^n \assuoneeta  \vol (B^n) 
     \text{ for all } x \in K \text{ and } 0 < \delta \leq \assuoneepsilon.
     \end{equation}
\end{assu}
In other words, Assumption~\ref{ASSU_one} claims that $\mainset$ contains a constant fraction $\eta_K$ of the full ball $\ball{\delta}{x}$ around $x$ for any radius $\delta > 0$ small enough. 
This rather mild assumption is discussed in some detail in \cite{deKlerk2017}. In particular, it is implied by the so-called {\em interior cone condition} used in approximation theory; it is satisfied by convex bodies and, more generally, by sets that are star-shaped with respect to a ball.

\begin{thm}
\label{THM_bound_interior_cone}
Let $\mainset \subseteq \R^n$ be a compact set satisfying Assumption~\ref{ASSU_one}. Then we have
\begin{equation}
    \ubError{r}{}{\mainfunc} = \bigO\bigg(\frac{\log r}{r}\bigg).
\end{equation}
\end{thm}

\begin{thm}
\label{THM_bound_convex_body}
Let $\mainset \subseteq \R^n$ be a convex body. Then we have
\begin{equation}
    \ubError{r}{}{\mainfunc} = \bigO\bigg(\frac{\log^2 r}{r^2}\bigg).
\end{equation}
\end{thm}

\medskip
\noindent
\textbf{Outline of the proofs.} First of all, if $f$ has a global minimizer  which  lies in the interior of $\mainset$, then we may apply Theorem~\ref{THM_interiorglobalminimizer} to obtain a convergence rate in $\bigO(1/r^2)=\bigO((\log r/r)^2)$ and so there is nothing  to prove. Hence, in the rest of the section, we  assume that $f$ has  a global minimizer  which lies on the boundary of $K$.

The basic proof strategy for both theorems is to construct explicit sum-of-squares polynomials $q_r$ giving good feasible solutions to the program \eqref{EQ_fmin_measures_bounded}.
The building blocks for these polynomials $\Laspolye{r}$ will be provided by the \emph{needle polynomials} from \cite{Kroo1992}; these are degree $r$ univariate polynomials $\needle{r}{h}, \halfneedle{r}{h}$, parameterized by a constant $h \in (0,1)$, that approximate well the Dirac delta at  $0$ on $[-1, 1]$ and $[0, 1]$, respectively.

For Theorem~\ref{THM_bound_interior_cone}, we are able to use the needle polynomials $\needle{r}{h}$ directly after applying the transform $x \mapsto \|x\|$ and selecting the value $h = h(r)$ carefully. We then make use of Lipschitz continuity of $f$ to bound the integral in the objective of \eqref{EQ_fmin_measures_bounded}.

For Theorem~\ref{THM_bound_convex_body}, a more complicated analysis is needed. 
We then construct $\Laspolye{r}$ as a product of $n$ univariate well-selected needle polynomials, exploiting geometric properties  of the boundary of $K$ in the neighbourhood of a global minimizer.

\medskip 
\noindent
\textbf{Simplifying assumptions.}
In order to simplify notation in the subsequent proofs we assume throughout this section that $0 \in K \subseteq B^n \subseteq \R^n$, and $\mainfuncmin  = f(0) = 0$, so $a=0$ is a global minimizer of $f$ over $K$.  As $\mainset$ is compact, and in light of Lemma~\ref{LEM_lineartransformation}, this is without loss of generality.

We now introduce needle polynomials and their main properties in Section~\ref{SEC_needles}, and then  give the proofs of \mbox{Theorems \ref{THM_bound_interior_cone} and \ref{THM_bound_convex_body}} in \mbox{Sections \ref{SEC:THM:ASSU} and \ref{SEC:THM:BODY}}, respectively.

\subsection{Needle polynomials}
\label{SEC_needles}

We begin by recalling some of the basic properties of the Chebyshev polynomials.
The Chebyshev polynomials  $\Cheby{r} \in \R[t]_r$ can be defined by the recurrence relation \eqref{eq:3term}, and also by the following explicit expression:
\begin{equation}
    \label{EQ_Chebyexplicit}
    \Cheby{r}(t) = \begin{cases} \cos (r \arccos t) &\text{for } |t| \leq 1, \\ \frac{1}{2}(t + \sqrt{t^2 -1})^r + \frac{1}{2}(t - \sqrt{t^2 -1})^{r} &\text{for } |t| \ge 1. \end{cases}
\end{equation}
From this definition, it can be seen that $|\Cheby{r}(t)| \leq 1$ on the interval $[-1, 1]$, and that $T_r(t)$ is nonnegative and monotone nondecreasing on  $[1, \infty)$.
The Chebyshev polynomials form an orthogonal basis of $\R[t]$ with respect to the Chebyshev measure (with weight $(1-t^2)^{-1/2}$) on $[-1,1]$ and they are used extensively in approximation theory. For instance, they are the polynomials attaining equality in the Markov brother's inequality on $[-1, 1]$, recalled below.

\begin{lem}[Markov Brothers' Inequality; see, e.g., \cite{Shadrin2005}]
\label{LEM_markov}
Let $p\in \mathbb R[t]$ be a univariate polynomial of degree at most $r$. Then, for any scalars $a< b$,  we have
\begin{equation}
    \max_{t \in [a, b]}|p'(t)| \leq \frac{2r^2}{b - a} \cdot \max_{t \in [a,b]}|p(t)|.
\end{equation}
\end{lem}
Kro\'o and Swetits \cite{Kroo1992} use the Chebyshev polynomials to construct the so-called (univariate) \emph{needle polynomials}.
\begin{defn}
For $r \in \N, h \in (0, 1)$, we define the {\em needle polynomial} $\needle{r}{h} \in \mathbb R[t]_{4r}$ by
\begin{equation}
\label{EQ_needle_polynomial}
\needle{r}{h}(t) = \frac{\Cheby{r}^2(1 + h^2 - t^2)}{\Cheby{r}^2(1 + h^2)}.
\end{equation}
Additionally, we define the {\em $\frac{1}{2}$-needle polynomial} $\halfneedle{r}{h} \in \mathbb R[t]_{4r}$ by
\begin{equation}
    \halfneedle{r}{h}(t) = \Cheby{2r}^2\bigg(\frac{2 + h - 2t}{2 - h}\bigg) \cdot \Cheby{2r}^{-2}\bigg(\frac{2+h}{2-h}\bigg).
\end{equation}
\end{defn}
By construction, the needle polynomials $\needle{r}{h}$ and $\halfneedle{r}{h}$ are squares and have degree $4r$. 
They approximate well the Dirac delta function at $0$ on $[-1, 1]$ and $[0, 1]$, respectively. In \cite{Sendov}, a construction similar to the needles presented here is used to obtain the best polynomial approximation of the Dirac delta in terms of the Hausdorff distance.

\begin{figure}
    \centering
    \includegraphics[width = 0.8\textwidth]{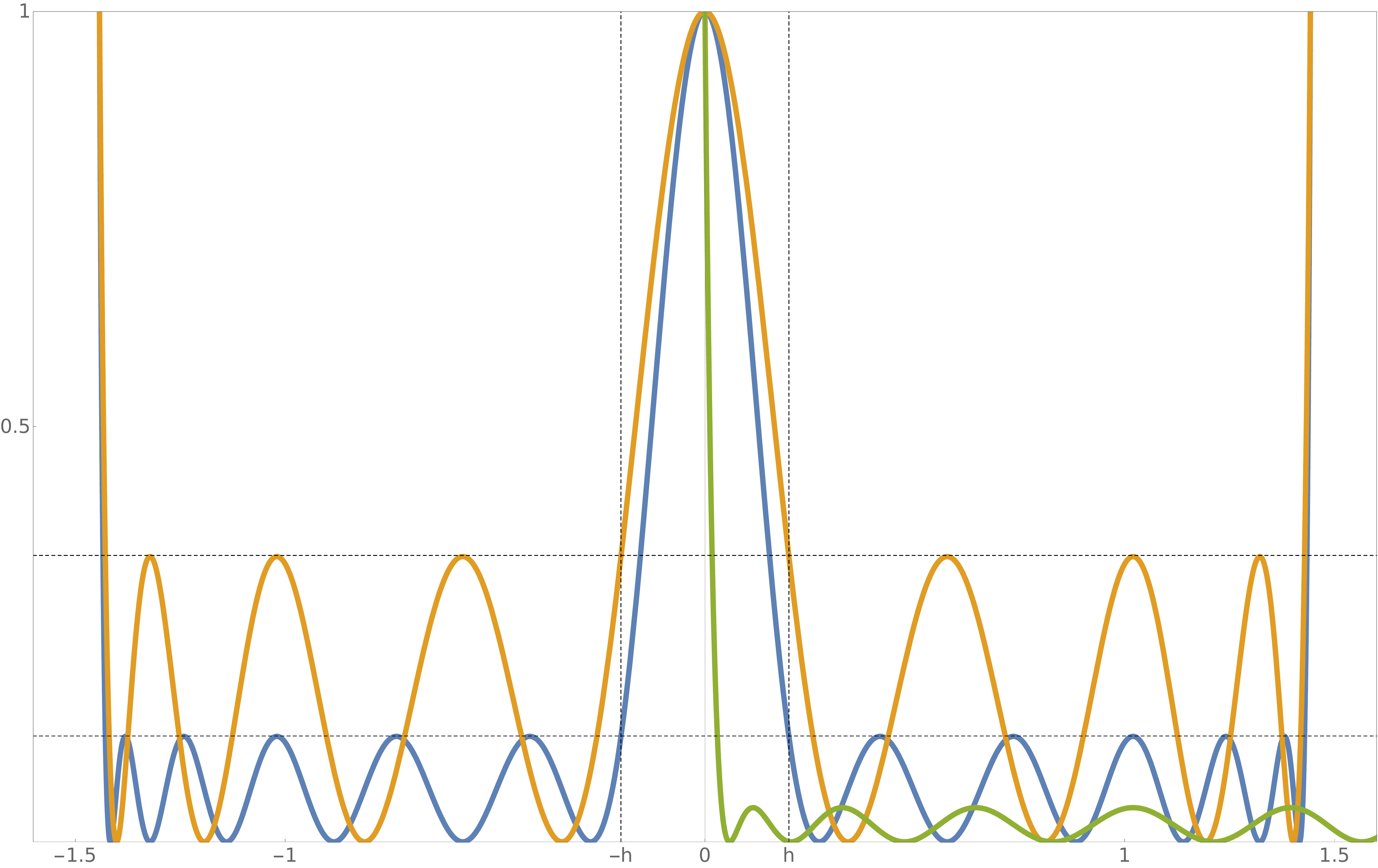}
    \caption{The needle polynomials $\needle{4}{h}$ (orange), $\needle{6}{h}$ (blue) and the 
    $\frac{1}{2}$-needle $\halfneedle{4}{h^2}$ (green) for $h = 1/5.$}
    \label{FIG_standard_needle}
\end{figure}

The needle polynomials  satisfy the following bounds (see Figure~\ref{FIG_standard_needle} for an illustration).

\begin{thm}[cf. \cite{Kroo1992, KrooLubinsky2012, Kroo2015}]
\label{THM_needle_properties}
For any $r \in \N$ and $h \in (0, 1)$, the following properties  hold for  the polynomials $\needle{r}{h}$ and $\halfneedle{r}{h}$:
\begin{align}
\label{EQ_needle_0}
    &\needle{r}{h}(0) = 1, \\
 \label{EQ_needle_1}
     0 \leq &\needle{r}{h}(t) \leq 1 &&\text{for } t \in [-1, 1], \\
    \label{EQ_needle_small}
    &\needle{r}{h}(t) \leq 4 e^{-{1\over 2} rh} &&\text{for } t \in [-1 , 1] \text{ with } |t| \ge h, 
\\
\\
\label{EQ_12needle_0}
    &\halfneedle{r}{h}(0) = 1, \\
  \label{EQ_12needle_1}
      0 \leq &\halfneedle{r}{h}(t) \leq 1 &&\text{for } t \in [0, 1], \\
    \label{EQ_12needle_small}
    &\halfneedle{r}{h}(t) \leq 4 e^{-{1\over 2}r\sqrt{h}} &&\text{for } t \in [0, 1] \text{ with } t \ge h.
\end{align}
\end{thm}
As this result plays a central role in our treatment we give a short proof, following the argument given in \cite{LasserrePauwels2017}. We need the following lemma.

\begin{lem}
\label{LEM_Cheby_bound}
For any $r \in \N$, $t \in [0, 1)$ we have
$
    \Cheby{r}(1 + t) \geq \frac{1}{2}e^{r \sqrt{t} \log (1 + \sqrt{2}) }\ge {1\over 2} e^{{1\over 4}r\sqrt t}.
$
\end{lem}
\begin{proof}
Using the explicit expression \eqref{EQ_Chebyexplicit} for $\Cheby{r}$, we have
\begin{align*}
        2\Cheby{r}(1 + t) 
        &\geq \bigg(1 + t + \sqrt{(1 + t)^2 -1}\bigg)^r =(1 + t + \sqrt{2 t + t^2})^r
                        \\
        &\geq (1 + \sqrt{2t})^r 
        =e^{r \log (1 +  \sqrt {2}\cdot \sqrt{t}) }.
\end{align*}
By concavity of the logarithm, we have
    \begin{align*}
        \log(1 + \sqrt{2} \sqrt{t}) &= \log (\sqrt{t} \cdot (1 + \sqrt{2}) + (1 - \sqrt{t})\cdot 1) \\
                                    &\geq \sqrt{t} \cdot \log (1 + \sqrt{2}) + (1-\sqrt{t})\log(1) = \sqrt{t} \cdot \log (1 + \sqrt{2})\ge {1\over 4}\sqrt t,
    \end{align*}
    and so, using the above lower bound on $T_r(1+t)$, 
    we obtain
    \begin{equation*}
        \Cheby{r}(1 + t) \geq \frac{1}{2} e^{ r\sqrt{t} \log (1 + \sqrt{2})} \ge \frac{1}{2} e^{{1\over 4} r\sqrt{t}}.
\end{equation*}
\qedMP
\end{proof}

\begin{proof}[of Theorem~\ref{THM_needle_properties}]
    Properties \eqref{EQ_needle_0}, \eqref{EQ_12needle_0} are clear.
    We first check \eqref{EQ_needle_1}-\eqref{EQ_needle_small}. If $|t|\le h$ then $1+h^2\ge 1+h^2-t^2\ge 1$, giving $\needle{r}{h}(t)\leq \needle{r}{h}(0) =  1$ by monotonicity of $T_r(t)$ on $[1,\infty)$.
    Assume now $h\le |t|\le 1$. Then $T_r^2(1+h^2-t^2)\le 1$ as $1+h^2-t^2\in [-1,1]$, and $T_r^2(1+h^2)\ge 1$ (again by monotonicity), which implies $\needle{r}{h}(t)\le 1$.
    In addition, since $T_r(1+h^2)\ge {1\over 2} e^{{1\over 4}rh}$ by Lemma~\ref{LEM_Cheby_bound}, we obtain
    $\needle{r}{h}(t) \le T_r^{-2}(1+h^2)\le 4 e^{-{1\over 2}rh}$.
    
    We now check  \eqref{EQ_12needle_1}-\eqref{EQ_12needle_small}.
    If $t\in [0,h]$ then $\halfneedle{r}{h}(t) \leq \halfneedle{r}{h}(0) = 1$ follows by monotonicity of $T_{2r}(t)$ on $[1,\infty).$
    Assume now $h\le t\le 1$. Then, ${2+h-2t\over 2-h}\in [-1,1]$ and thus $T_{2r}^2\big({2+h-2t\over 2-h}\big) \le 1$.
    On the other hand, we have $T_{2r}^2\big({2+h\over 2-h}\big) \ge 1$, which gives $\halfneedle{r}{h}(t)\le 1$.
    In addition,
    as ${2+h\over 2-h}\ge 1+h \ge 1$, using again monotonicity of $T_{2r}$ and Lemma~\ref{LEM_Cheby_bound}, we get 
    $T_{2r}^2\big({2+h\over 2-h}\big) \ge T_{2r}^2(1+h)\ge {1\over 4}e^{{1\over 2}r\sqrt h}$, which implies \eqref{EQ_12needle_small}.
        \qedMP      
\end{proof}

\medskip\noindent
We now give 
 a simple  lower estimator for a nonnegative polynomial $p$ with $p(0) = 1$. This lower estimator  will be useful later to lower bound the integral of the needle and $\frac{1}{2}$-needle polynomials on small intervals $[-h, h]$ and $[0, h]$, respectively.

\begin{lem}
\label{LEM_needlevolume}
Let $p\in \R[t]_r$ be a   polynomial, which is nonnegative over $\R_{\geq 0}$ and satisfies  $p(0)=1$,  $p(t) \leq 1$ for all $t\in [0,1]$. 
Let $\needlelb{r} : \R_{\geq 0} \rightarrow \R_{\geq 0}$ be defined by
\begin{equation}
\needlelb{r}(t) = \begin{cases}
            1 - 2r^2 t &\quad \text{ if } t \leq \frac{1}{2r^2},\\
            0 &\quad \text{ otherwise}.
       \end{cases}   
\end{equation}
Then $\needlelb{r}(t) \leq p(t)$ for all $t \in \R_{\geq 0}$.
\end{lem}

\begin{proof}
Suppose not. Then there exists $s \in \R_{\geq 0}$ such that $\needlelb{r}(s) > p(s)$. As $p \geq 0$ on $\R_{\ge 0}$, $p(0)=1$  and $\needlelb{r}(t) = 0$ for $t \ge \frac{1}{2r^2}$, we have $0<s < \frac{1}{2r^2}$. We find that
$p(s) - p(0) < \needlelb{r}(s) - 1 =  -2r^2s.$
Now, by the mean value theorem, there exists an element $z \in (0, s)$ such that
$p'(z) = \frac{p(s) - p(0)}{s} < \frac{-2r^2s}{s} = -2r^2$.
But this is in contradiction with Lemma~\ref{LEM_markov}, which implies that $\max_{t \in [0,1]}|p'(t)| \leq 2r^2.$
\qedMP\end{proof}

\begin{cor}\label{COR:Lambda}
Let $h \in (0,1)$, and let $\needle{r}{h}, \halfneedle{r}{h}$ as above. Then $\needlelb{4r}(t) \leq \needle{r}{h}(t) = \needle{r}{h}(-t)$ and $\needlelb{4r}(t) \leq \halfneedle{r}{h}(t)$ for all $t \in [0, 1]$.
\end{cor}

\subsection{Compact sets satisfying Assumption \ref{ASSU_one}}\label{SEC:THM:ASSU}

In this section we prove  Theorem~\ref{THM_bound_interior_cone}. Recall  we assume that $K$ satisfies  Assumption~\ref{ASSU_one} with constants $\assuoneepsilon$ and $\assuoneeta$. We also assume that $0\in \partial K$ is a global minimizer of $f$ over $K$, $f(0)=0$, and $K\subseteq B^n$, so that $\epsilon_K<1$.
By Lemma~\ref{LEM_Lipschitzupperestimator}, we have 
$\mainfunc(x) \leq_0 \gradConst{\mainset}{\mainfunc} \|x\|$ on $\mainset$. Hence, in view of  Lemma~\ref{LEM_upperestimator},   it suffices to find a polynomial $\Laspolye{r} \in \Sigma_{2r}$ for each $r \in \N$ such that $\volu{\mainset}{\Laspolye{r}} = 1$ 
and
\begin{equation}
    \int_\mainset \Laspolye{r}(x) \|x\| dx = \bigO\bigg( \frac{\log r}{r} \bigg).
\end{equation}
The idea is to set $\Laspolye{r}(x) \sim \absneedle{r}{h}(x) := \needle{r}{h}(\|x\|)$ and then  select carefully the constant $h = h(r)$.
The main technical component of the proof is the following lemma, which bounds the normalized integral $\int_\mainset \absneedle{r}{h}(x)\|x\|^\beta dx$ in terms of $r, h$ and $\beta \geq 1$. For Theorem~\ref{THM_bound_interior_cone} we only need the case $\beta=1$, but allowing $\beta\ge 1$ permits to show a sharper convergence rate when the polynomial $f$ has special properties at the minimizer (see Theorem~\ref{THM_fasterconvergence}).

\begin{lem}
\label{LEM_interior_cone_technical}
Let $r \in \N$ and $h \in (0,1)$ with $\assuoneepsilon \geq h \geq {1}/{64r^2}$. Let $\beta \geq 1$. Then
\begin{equation}
\label{EQ_interior_cone_technical}
    \frac{1}{\volu{\mainset}{\absneedle{r}{h}}} \int_\mainset \absneedle{r}{h}(x) \|x\|^\beta dx \leq h^\beta + C r^{2n} e^{-{1\over 2}hr},
\end{equation}
where $C > 0$ is a constant depending only on $\mainset$.
\end{lem}

\begin{proof}
Set $\rho = 1/ 64r^2$, so that $\rho\le h\le \epsilon_K$. We define the sets
\begin{equation}\label{EQ:ball}
    \boxl := \ball{h}{0} \cap K \text{ and } \boxs := \ball{\rho}{0} \cap K \subseteq \boxl.
\end{equation}
Note that $\vol(\boxl) \geq \vol(\boxs) \geq \assuoneeta \rho^n \vol (B^n)$ by Assumption~\ref{ASSU_one}. For $x\in \boxl$, we have the   bounds $\absneedle{r}{h}(x) \leq 1$ (by \eqref{EQ_needle_1}, since $\|x\|\le 1$ as $K\subseteq B^n$) and  $\|x\|^\beta \leq h^\beta$. On the other hand,  for $x\in \mainset \setminus \boxl$, we have the  bound $\|x\|^\beta \leq 1$, but now $\absneedle{r}{h}(x)$ is exponentially small (by \eqref{EQ_needle_small}).
We exploit this for bounding the integral in \eqref{EQ_interior_cone_technical}:
\begin{align*}
    \nonumber \int_\mainset \absneedle{r}{h}(x) \|x\|^\beta dx 
    &=   \int_{\boxl} \absneedle{r}{h}(x)\|x\|^\beta dx + \int_{\mainset \setminus \boxl} \absneedle{r}{h}(x)\|x\|^\beta dx \\
    &\leq  h^\beta \int_{\boxl} \absneedle{r}{h}(x)  dx + \int_{\mainset \setminus \boxl} \absneedle{r}{h}(x)dx.   
\end{align*}
Combining with the following lower bound on the denominator:
$$ \int_K \absneedle{r}{h}(x) dx\ge \int_{B_h} \absneedle{r}{h}(x) dx \ge  \int_{B_\rho} \absneedle{r}{h}(x) dx,$$
we get
\begin{equation*}  \label{EQ_ALIGN_interior_cone_technical}
      \frac{1}{\volu{\mainset}{\absneedle{r}{h}}} \int_\mainset \absneedle{r}{h}(x) \|x\|^\beta dx  
 \ \le \ 
  h^\beta+     \frac{\int_{\mainset\setminus B_h}  \absneedle{r}{h}(x) dx}   {\volu{ B_\rho}{\absneedle{r}{h}}}.
\end{equation*}
It remains to upper bound the last term in the above expression. 
By \eqref{EQ_needle_small} we have $ \absneedle{r}{h}(x) \leq 4e^{-{1\over 2}hr}$ for any $x\in K\setminus B_h$ and so  
\begin{equation}
\volu{\mainset \setminus \boxl}{\absneedle{r}{h}} \leq 4e^{-{1\over 2}hr} \cdot \vol(K \setminus \boxl) \leq 4e^{-{1\over 2}hr} \cdot \vol (B^n).
\end{equation}
Furthermore, by Lemma~\ref{LEM_needlevolume}, we have $ \absneedle{r}{h}(x) \ge \needlelb{4r}(\|x\|) =1- 32 r^2\|x\|\ge {1\over 2}$ for all $x \in B_\rho$.
Using Assumption~\ref{ASSU_one} we obtain
\begin{equation}
    \label{EQ_hballintegrallowerbound}
   \int_{B_{\rho}} {\absneedle{r}{h}} (x)dx
    \geq \frac{1}{2} \vol(\boxs) \geq \frac{1}{2}\assuoneeta\rho^n \vol(B^n)= {\eta_K \vol(B^n)\over 2\cdot 64^n r^{2n}}.
\end{equation}
Putting things together yields
\begin{equation}
\frac{\volu{\mainset \setminus \boxl}{\absneedle{r}{h}}}{\volu{\boxl}{\absneedle{r}{h}}} 
\leq 4e^{-{1\over 2}hr} \cdot \vol (B^n) {2 \cdot 64^n r^{2n} \over \eta_K \vol(B^n)}
= {8\cdot 64^n \over \eta_K} r^{2n} e^{-{1\over 2} hr}.
\end{equation}
This shows the lemma with the constant $C={8\cdot 64^n\over \eta_K}$.
\qedMP \end{proof}
It remains to choose $h = h(r)$ to obtain the polynomials $\Laspolye{r}$. Our choice here is essentially the same as the one used in \cite{Kroo2015, Sendov}.
With the next result (applied with $\beta = 1$) the proof of Theorem~\ref{THM_bound_interior_cone} is now complete.

\begin{prop}
\label{PROP_bound_interior_cone_univariate}
For $r\in \N$ and $\beta\ge 1$, set $h(r) =2(2n+\beta)\log r/r$ 
and  define the polynomial $\Laspolye{r} := \absneedle{r}{h(r)} / \volu{\mainset}{\absneedle{r}{h(r)}}$. Then $q_r$  is a sum-of-squares polynomial of degree $4r$ with $\volu{\mainset}{\Laspolye{r}} = 1$ and 
\begin{equation}
    \int_\mainset \Laspolye{r}(x)\|x\|^\beta dx = \bigO\bigg(\frac{\log^\beta r}{r^\beta}\bigg).
\end{equation}
\end{prop}

\begin{proof}
For $r$ sufficiently large, we have $h(r) < \assuoneepsilon$ and $h(r) \geq {1}/{64r^2}$ and so we may use Lemma~\ref{LEM_interior_cone_technical} to obtain
\begin{align}
\int_\mainset \Laspolye{r}(x) \|x\|^\beta dx &\leq h(r)^\beta + Cr^{2n} e^{-{1\over 2}h(r)r} \\ 
&= \bigg(2(2n+\beta)\frac{ \log r} {r}\bigg)^\beta + {C\over r^\beta}
= \bigO\bigg(\frac{\log^\beta r}{r^\beta}\bigg).
\end{align}
\qedMP\end{proof}

\subsection{Convex bodies}\label{SEC:THM:BODY}

We now prove Theorem~\ref{THM_bound_convex_body}. Here, $\mainset$ is assumed to be a convex body, hence it still satisfies Assumption~\ref{ASSU_one} for certain constants $\assuoneepsilon, \assuoneeta$. As before we also assume that $0\in \partial K$ is a global minimizer of $f$ in $K$, $f(0)=0$  and $K\subseteq B^n$.

If $\nabla \mainfunc (0) = 0$, then in view of Taylor's theorem (Theorem~\ref{THM_Taylor}) we know that $f(x)\le_0 \gamma_{K,f}\|x\|^2$ on $K$.  
Hence we may  apply Proposition~\ref{PROP_bound_interior_cone_univariate} (with $\beta=2$) to this quadratic upper estimator of $f$  to obtain $\ubError{r}{}{\mainfunc} = \bigO(\log^2 r / r^2)$ (recall Lemma~\ref{LEM_upperestimator}).

In the rest of this section, we will therefore assume that $\nabla \mainfunc(0) \neq 0$. In this case, we cannot get a better upper estimator than $f(x) \leq_0 \beta_{K,f}\|x\|$ on $K$, and so the choice of $\Laspolye{r}$ in Proposition~\ref{PROP_bound_interior_cone_univariate} is not sufficient. Instead we will need to make use of the sharper $\frac{1}{2}$-needles $\halfneedle{r}{h}$. We will show how to do this in the univariate case first.

\medskip
\noindent
\textbf{The univariate case.}
If $\mainset \subseteq [-1,1]$ is  convex  with $0$ on its boundary, we may assume w.l.o.g. that $\mainset = [0, b]$ for some $b \in (0, 1]$ (in which case we may choose $\epsilon_K=b$).
 By using the $\frac{1}{2}$-needle $\halfneedle{r}{h}$  instead of the regular needle $\needle{r}{h}$, we immediately get the following analog of Lemma~\ref{LEM_interior_cone_technical}.

\begin{lem}
\label{LEM_convex_body_technical_univariate}
Let $b \in (0, 1]$ and $\mainset = [0, b]$. Let $r \in \N$ and $h \in (0,1)$ with $b \geq h \geq \frac{1}{64r^2}$. Then we have
\begin{equation}
\label{EQ_convex_body_technical_univariate}
    \frac{1}{\volu{\mainset}{\halfneedle{r}{h}}} \int_\mainset \halfneedle{r}{h}(x) |x| dx \leq h + C r^{2} e^{-{1\over 2}\sqrt{h}r},
\end{equation}
where $C > 0$ is a universal constant.
\end{lem}

\begin{proof}
Same proof as for Lemma~\ref{LEM_interior_cone_technical}, using now the fact that $\halfneedle{r}{h}(x) \leq 1$ on $K$ and $ \halfneedle{r}{h}(x) \leq 4e^{-{1\over 2}\sqrt{h}r}$ on $K\setminus B_h$ from \eqref{EQ_12needle_1} and  \eqref{EQ_12needle_small}.
\qedMP\end{proof}

Since the exponent in \eqref{EQ_convex_body_technical_univariate} now contains the term `$\sqrt{h}$' instead of  `$h$', we may square our previous choice of $h(r)$ in Proposition~\ref{PROP_bound_interior_cone_univariate} to obtain the following result.

\begin{prop}
\label{PROP_bound_convex_body_univariate}
Assume $K = [0, b]$. Set $h(r) = \big(2{\log (r^4)\over r}\big)^2= \big(8{\log r\over r}\big)^2$ and define the polynomial $\Laspolye{r} := {\halfneedle{r}{h(r)}}/{\volu{\mainset}{\halfneedle{r}{h(r)}}}$.
Then $q_r$  is a sum-of-squares polynomial of degree $4r$ satisfying  $\volu{\mainset}{\Laspolye{r}} = 1$ and 
\begin{equation}
    \int_\mainset \Laspolye{r}(x) x dx = \bigO\bigg(\frac{\log^2 r}{r^2}\bigg).
\end{equation}
\end{prop}

\begin{proof}
For $r$ sufficiently large, we have $h(r) < b$ and $h(r) \geq {1}/{64r^2}$ and so we may use Lemma~\ref{LEM_convex_body_technical_univariate} to obtain
\begin{align}
\int_\mainset \Laspolye{r}(x) x dx \leq h(r) + Cr^{2} e^{-{1\over 2} r\sqrt{h(r)}} 
= \bigg(8{\log r\over r}\bigg)^2+ {C\over r^2}
= \bigO\bigg(\frac{\log^2 r}{r^2}\bigg).
\end{align}
\qedMP\end{proof}
Since $f(x)\leq_0 \beta_{K,f}\cdot x$ on $K$ we obtain
$\ubError{r}{}{\mainfunc}=\bigO((\log r/r)^2)$, the desired result.

\medskip\noindent
\textbf{The multivariate case.} 
Let $v := \nabla \mainfunc (0) / \|\nabla \mainfunc (0)\|$ and let $w_1, w_2, \dots w_{n-1}$ be an orthonormal basis of $v^\perp$. Then
\begin{equation}\label{eqU}
U = U(\mainfunc) := \{v, w_1, w_2, \dots, w_{n-1}\}
\end{equation}
is an orthonormal basis, which we will use as basis of $\R^n$.

 The basic idea of the proof is as follows. For any $j\in [n-1]$, if we minimize $f$ in the direction of $w_j$ then  we minimize the univariate polynomial $\tilde f(t)=f(tw_j)$, which satisfies: $\tilde f'(0)= \langle \nabla f(0), w_j\rangle=0$. Hence, by Taylor's theorem, there is a quadratic upper estimator when minimizing in the direction $w_j$, so that using a regular needle polynomial will suffice for the analysis.
On the other hand, if we minimize $f$ in the direction $v$, then $\min_{tv\in K}f(tv)=\min_{t\in [0,1]} f(tv)$,  since $K\subseteq B^n$ and $v\in N_K(0)$. As explained above this univariate minimization problem can be dealt with using  $\frac{1}{2}$-needle polynomials to get the desired convergence rate.
This motivates defining the following sum-of-squares polynomials.
\begin{defn}\label{DEF:multineedle}
For $r \in \N, h \in (0, 1)$ we define the polynomial $\multineedle{r}{h} \in \Sigma_{2nr}$ by
\begin{equation}
	\multineedle{r}{h}(x) = \halfneedle{r}{h^2}(\langle x, v \rangle ) \cdot \prod_{j = 1}^{n-1} \needle{r}{h}(\langle x, w_j \rangle).
\end{equation}
\end{defn}
This construction is similar to the one used by Kro\'o in \cite{Kroo2015} to obtain sharp multivariate needle polynomials at boundary points of $\mainset$.
\begin{prop}
\label{PROP_multineedle}
We have    $\multineedle{r}{h}(0)=1$ and 
\begin{align}
    \multineedle{r}{h}(x) &\in [0,1] &&\text{for } x \in \mainset, \\
    \label{EQ_needlesmall_boxsmall}
    \multineedle{r}{h}(x) &\leq 4 e^{-{1\over 2}hr} &&\text{for } x \in \mainset \text{ with } \langle x, v \rangle \geq h^2, \\ 
    \label{EQ_needlesmall_boxlarge}
    \multineedle{r}{h}(x) &\leq 4 e^{-{1\over 2}hr} &&\text{for } x \in \mainset \text{ with } \max_{j \in [n-1]} |\langle x, w_j \rangle| \geq h.
\end{align}
\end{prop}
\begin{proof}
Note that for any $x \in \mainset$ we have $ 0\le \langle x, v \rangle \leq \|x\| \leq 1$ and $|\langle x, w_j \rangle| \leq \|x\| \leq 1$ for $j\in [n-1]$. The required properties then follow immediately from those of the needle and $\frac{1}{2}$-needle polynomials discussed in Theorem~\ref{THM_needle_properties}. 
\qedMP\end{proof}

It remains to formulate and prove an analog of Lemma~\ref{LEM_interior_cone_technical} for the polynomial $\multineedle{r}{h}$. Before we are able to do so, we first need a few technical statements.
For $h > 0$ we define the polytope
\begin{equation}
\multbox{h} := \{x \in \R^n : 0 \leq \langle x, v \rangle \leq h^2, |\langle x, w_j \rangle| \leq h \text{ for all } j \in [n-1] \}.
\end{equation}
 \noindent
Note that for $h \in (0, 1)$, the inequalities \eqref{EQ_needlesmall_boxsmall} and \eqref{EQ_needlesmall_boxlarge} can be summarized as
\begin{equation}    \label{EQ_needlesmall_on_multibox}
	\multineedle{r}{h}(x) \leq 4e^{-{1\over 2}hr} \text{ for } x \in \mainset \setminus \multbox{h},
\end{equation}
which means  $\multineedle{r}{h}(x)$ is exponentially small for $x\in K$ outside of $\multbox{h}$. When instead $x \in K\cap \multbox{h}$, the following two lemmas show that the function value $\mainfunc(x)$ is small.

\begin{lem}
\label{LEM_boxnormbound}
Let $h \in (0,1)$. Then $\|x\| \leq \sqrt{n} h$ for all $x \in  \multbox{h}$.
\end{lem}
\begin{proof}
Let $x \in \multbox{h}$. By expressing $x$ in the orthonormal basis $U$ from \eqref{eqU}, we obtain
\begin{equation}
    \|x\|^2 = \langle x, v \rangle^2 + \sum_{i =1}^{n-1}\langle x, w_i \rangle^2 \leq nh^2,
\end{equation}
using the definition of $\multbox{h}$ for the second inequality.
\qedMP\end{proof}

\begin{lem}
\label{LEM_Taylorbound}
Let $h \in (0, 1)$. Then
$\mainfunc(x) \leq \big(\gradConst{\mainset}{\mainfunc} + n\HessConst{\mainset}{\mainfunc}\big) h^2$ \ for all  $ x \in K\cap \multbox{h}.$
\end{lem}

\begin{proof}
Using Taylor's Theorem~\ref{THM_Taylor}, Lemma~\ref{LEM_boxnormbound} and $\langle x,v\rangle \le h^2$ for $x\in \multbox{h}$, we obtain
\begin{align}
    \mainfunc(x) \leq \langle \nabla \mainfunc (0), x \rangle + \HessConst{\mainset}{\mainfunc} \|x\|^2 \leq \|\nabla \mainfunc(0)\| \langle x, v \rangle + n\HessConst{\mainset}{\mainfunc}h^2 \leq \big(\gradConst{\mainset}{\mainfunc} + n\HessConst{\mainset}{\mainfunc}\big) h^2.
\end{align}
\qedMP \end{proof}
We now give a lower  bound on $\volu{\mainset \cap \multbox{h}}{\multineedle{r}{h}}$ (compare to \eqref{EQ_hballintegrallowerbound}). First we need the following bound on $\vol(\mainset \cap \multbox{h})$.

\begin{lem}
\label{LEM_multboxvolume}
Let $h \in (0,1)$. If $h < \assuoneepsilon$ then we have: 
  $ \vol(\mainset \cap \multbox{h}) \geq \assuoneeta h^{2n}\vol(B^n)$.
\end{lem}

\begin{proof}
Consider the halfspace $H_v := \{ x \in \R^n : \langle v, x \rangle \geq 0\}$.
As $v \in \normalcone{\mainset}{0}$, we have the inclusion $\mainset \subseteq H_v$. We show that $\ball{h^2}{0} \cap H_v \subseteq \multbox{h}$, implying that $\ball{h^2}{0} \cap \mainset \subseteq \ball{h^2}{0} \cap H_v \subseteq \multbox{h}$.
Let $x \in \ball{h^2}{0} \cap H_v$. By expressing $x$ in the orthonormal basis $U(f)$ from \eqref{eqU}, we get
$\|x\|^2=\langle v,x\rangle^2+\sum_{j=1}^{n-1}\langle w_j,x\rangle^2\le h^4$. Since $x\in H_v$ and $0<h<1$, we get   $0\le \langle v,x\rangle \le h^2$ and
$|\langle w_j,x\rangle |\le h^2\le h$, thus showing $x\in P_h$.
See Figure~\ref{FIG_multbox} for an illustration. 
We may now apply Assumption~\ref{ASSU_one} to find
$\vol(\multbox{h}) \geq \vol(\ball{h^2}{0} \cap \mainset) \geq \assuoneeta h^{2n}\vol(B^n).$
\qedMP
\end{proof}

\begin{figure}
    \centering
    \begin{tikzpicture}[x=5cm,y=5cm]
\begin{scope}
%
%
%
%
%

%
%
%

    \draw[thick] (0,0) -- (-0.4, 0.8) -- (0.4, 0.8) --cycle;
    \draw[opacity=0.1, fill] (0,0) -- (-0.4, 0.8) -- (0.4, 0.8) --cycle;
    
    \draw[dashed, thick] (0,0) -- (-0.6, 0) -- (-0.6, 0.4) -- (0.6, 0.4) -- (0.6, 0) --cycle;

	\begin{scope}
        \clip (0,0) -- (-0.6, 0) -- (-0.6, 0.4) -- (0.6, 0.4) -- (0.6, 0) --cycle;
        \draw[dashed] (0, 0) circle (2cm);
    \end{scope}
	\begin{scope}
        \clip (0,0) -- (-0.4, 0.8) -- (0.4, 0.8) --cycle;
        \draw[fill, opacity = 0.2] (0, 0) circle (2cm);
    \end{scope}

    \draw[] (0.6,0) node[below] {$\multbox{h}$};
    \draw[] (0.4, 0.8) node[below right] {$\mainset$};
    \draw[] (0, 0) node[below] {$0$};
    \draw[] (0.35, 0.2) node[right] {$\ball{h^2}{0}$};

    \draw[->, ultra thick]   (0, 0) -- node[above right] {$v$} (0, 0.25);
    \draw[->, ultra thick]   (0, 0) -- node[below] {$w_1$} (0.25, 0);
    \fill[red] (0,0) circle (2pt);
\end{scope}
\end{tikzpicture}
    \caption{Overview of the situation in the proof of Lemma~\ref{LEM_multboxvolume}. Note that as long as $v \in \normalcone{\mainset}{0}$, the entire region $\ball{h^2}{0} \cap \mainset$ (in dark gray) is contained in $\multbox{h}$.}
    \label{FIG_multbox}
\end{figure}
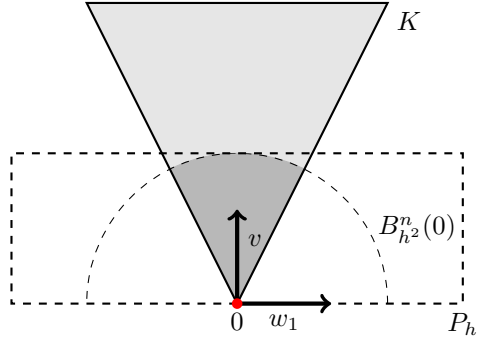

\begin{lem}
\label{LEM_box_polyvolume}
Let $r \in \N, h \in (0, 1)$. Assume that $\assuoneepsilon > h > \rho = {1}/{64r^2}$. Then 
\begin{equation}
    \volu{\mainset \cap \multbox{h}}{\multineedle{r}{h}} \geq \frac{1}{2^n}\assuoneeta \rho^{2n}\vol(B^n).
    \end{equation}
\end{lem}

\begin{proof}
The integral $  \volu{\mainset \cap  \multbox{h}}{\multineedle{r}{h}}$ is equal to 
\begingroup
\allowdisplaybreaks
\begin{align}
 &    \int_{\mainset \cap  \multbox{h}} \halfneedle{r}{h^2}(\langle x, v \rangle) \cdot \prod_{i = 1}^{n-1} \needle{r}{h}(\langle x, w_i \rangle) dx 
       & \text{\rm [using Definition \ref{DEF:multineedle}]} \\
&    \geq \int_{\mainset \cap  \multbox{h}} \needlelb{4r}(\langle x, v \rangle) \cdot \prod_{i = 1}^{n-1} \needlelb{4r}(|\langle x, w_i \rangle|)dx 
       &\text{\rm [using Corollary \ref{COR:Lambda}]} \\
&    \geq \int_{\mainset \cap \multbox{\rho}} \needlelb{4r}(\langle x, v \rangle) \cdot \prod_{i = 1}^{n-1} \needlelb{4r}(|\langle x, w_i \rangle|)dx 
        & \text{\rm [since } \multbox{\rho} \subseteq \multbox{h} \text{\rm ]}\\
 &   \geq \int_{\mainset \cap 
        \multbox{\rho}} \needlelb{4r}(\rho) \cdot \prod_{i = 1}^{n-1} \needlelb{4r}(\rho)dx 
       & \text{\rm [since } \Lambda_{4r}(t)\ge \Lambda_{4r}(\rho)\ \text{\rm if } t\in [0,\rho]   \text{\rm ]} \\
&    \geq \int_{\mainset \cap \multbox{\rho}} \frac{1}{2^n} dx 
       = \frac{1}{2^n}  \vol(\mainset \cap \multbox{\rho}) & 
        \text{\rm [as } \Lambda_{4r}(\rho)={1\over 2} \text{]}\\
&     \ge {1\over 2^n} \eta_K \rho^{2n}\vol(B^n) & \text{\rm [using Lemma \ref{LEM_multboxvolume}]}
\end{align}%
\endgroup
\qedMP\end{proof}

We are now able to prove an analog of Lemma~\ref{LEM_interior_cone_technical}.
\begin{lem}
\label{LEM_convex_body_technical_multivariate}
Let $r \in \N$ and  $h \in (0,1)$. If $\assuoneepsilon > h >  {1}/{64r^2}$ then we have
\begin{equation}
\frac{1}{\volu{\mainset}{\multineedle{r}{h}}} \int_\mainset \multineedle{r}{h}(x) \mainfunc(x)  dx \leq (\beta_{K,f}+n\gamma_{K,f}) h^2 +
 C'r^{4n}  e^{-{1\over 2}hr},
  \end{equation}
where $C'$ is a constant depending only on $K$.
\end{lem}

\begin{proof}
Set $\rho={1}/{64r^2}$.
By Lemma~\ref{LEM_Taylorbound}, $f(x)\le (\beta_{K,f}+n\gamma_{K,f})h^2$ for all $x\in K\cap \multbox{h}$. Moreover, by Proposition~\ref{PROP_multineedle}, we have $\sigma^h_r(x)\le 4e^{-{1\over 2}hr}$ for all $x\in K\setminus \multbox{h}$. 
Hence,
\begin{align}
    \int_\mainset \multineedle{r}{h}(x) \mainfunc(x) dx 
    & =  \int_{\mainset \cap  \multbox{h}} \multineedle{r}{h}(x) \mainfunc(x) dx + \int_{\mainset \setminus  \multbox{h}} \multineedle{r}{h}(x) \mainfunc(x) dx \\
    & \le (\beta_{K,f}+n\gamma_{K,f})h^2 \int_{K\cap \multbox{h}} \sigma^h_r(x)dx + 4 e^{-{1\over 2}hr} f_{\max,K} \vol(B^n),
\end{align}
where $f_{\max,K}=\max_{x\in K}f(x)$.
Combining with 
$$\int_K\sigma^h_r(x)dx\ge \int_{K\cap \multbox{h}}\sigma^h_r(x)dx\ge {1\over 2^n} \eta_K\rho^{2n}\vol(B^n),$$
where we use Lemma~\ref{LEM_box_polyvolume} for the last inequality, we obtain
$$
    \nonumber\frac{1}{\volu{\mainset}{\multineedle{r}{h}}} \int_\mainset \multineedle{r}{h}(x) \mainfunc(x) dx 
\leq
    (\beta_{K,f}+n\gamma_{K,f})h^2 + {4\cdot 2^n \cdot 64^{2n} f_{\max,K}\over \eta_K }r^{4n} e^{-{1\over 2}hr}.
$$
This shows the lemma, with the constant $C'= {4\cdot 2^n \cdot 64^{2n} f_{\max,K}\over \eta_\mainset}$.
\qedMP\end{proof}
From the preceding lemma we get the following corollary, which immediately implies Theorem~\ref{THM_bound_convex_body}.
\begin{cor}
For  any $r \in \N$, set $h(r)= (8n+4){\log r\over r}$ and consider the polynomial 
$\Laspolye{r} :=  {\multineedle{r}{h(r)}}/{\volu{\mainset}{\multineedle{r}{h(r)}}}$.
Then $q_r$  is a sum-of-squares polynomial of degree $4nr$, which satisfies  $\volu{\mainset}{\Laspolye{r}} = 1$ and 
\begin{equation}
\int_\mainset \Laspolye{r}(x) \mainfunc(x) dx = \bigO\bigg(\frac{\log^2 r}{r^2}\bigg).
\end{equation}
\end{cor}

\begin{proof}
For $r$ sufficiently large, we have $\assuoneepsilon > h(r) > {1}/{64r^2}$ and so we may apply Lemma~\ref{LEM_convex_body_technical_multivariate}, which implies directly
\begin{align}
 \int_\mainset \Laspolye{r}(x) \mainfunc(x) dx 
    &\leq (\beta_{K,f}+n\gamma_{K,f}) h(r)^2 + 
    C'r^{4n} e^{-{1\over 2}rh(r)}       =\bigO\bigg({\log^2r\over r^2}\bigg).
    \end{align}
\qedMP\end{proof}
\section{Numerical Experiments}
	\label{SEC_NumericalExperiments}
	    In this section, we illustrate some of the results in this paper with numerical examples. We consider the test functions listed below in Table \ref{TAB_test_functions}, the latter four of which are well-known in global optimization and also used for this purpose in \cite{deKlerk2017}.
\begin{table}[h!]
\centering
\def\arraystretch{1.3}
\begin{tabular}{|l|l|l|}
\hline
Name & Formula & $f_{\min, [-1, 1]^2}$ \\
\hline
Linear 		& $f_{li}(x) = x_1$ & $f_{li}(-1, 0) = -1$ \\
Quadratic 	& $f_{qu}(x) = x_1 + x_2^2$ & $f_{qu}(-1,0) = -1$ \\
Booth  		& $f_{bo}(x) = (10x_1 + 20x_2 - 7)^2 + (20x_1 + 10x_2 -5)^2$ 			 	& $f_{bo}(\frac{1}{10},\frac{3}{10}) = 0$ \\
Matyas  	& $f_{ma}(x) = 26(x_1^2 + x_2^2) - 48x_1x_2$ 			  			     	& $f_{ma}(0,0) = 0$ \\
Camel  		& $f_{ca}(x) = 50x_1^2 -\frac{2625}{4}x^4_1 + \frac{15625}{6}x_1^6 + 25x_1x_2 + 25x_2^2$	 & $f_{ca}(0,0) = 0$ \\
Motzkin  	& $f_{mo}(x) = 64x_1^4x_2^2 + 64x_1^2x_2^4 - 48 x_1^2x_2^2 + 1$ 				& $f_{mo}(\pm \frac{1}{2}, \pm \frac{1}{2}) = 0$ \\
\hline
\end{tabular}
\caption{Polynomial test functions.}
\label{TAB_test_functions}
\end{table}
We compare the behaviour of the error $\ubError{r}{\mainset}{f}$ for these functions on different sets $\mainset$, namely the hypercube, the unit ball, and a regular octagon in $\R^2$. On the unit ball and the regular octagon, we consider the Lebesgue measure. On the hypercube, we consider both the Lebesgue measure and the Chebyshev measure. In each case, we compute the Lasserre bounds of order $r$ in the range  $1 \leq r \leq 20$, corresponding to sos-densities of degree up to $40$.

\medskip
\noindent \textbf{Computing the bounds.}
As explained in Section \ref{SEC_Introduction}, it is possible to compute the degree $2r$ Lasserre bound $\Lasupboundl{r}{K, \mu}$ by finding the smallest eigenvalue of the truncated moment matrix $M_{r, f}$ of $\mainfunc$, defined by
\[
	M_{r, f}(\alpha, \beta) = \int_\mainset f p_\alpha p_\beta d \mu(x) \quad (\alpha, \beta \in \N_r^n),
\]
assuming that one has an orthonormal basis $\{p_\alpha : \alpha \in \N_r^n \}$ of $\R[x]_r$ w.r.t. the inner product induced by the measure $\mu$, i.e., such that 
$ \int_\mainset p_\alpha p_\beta d \mu(x)=\delta_{\alpha,\beta}$.

More generally, if we use an arbitrary linear basis $\{p_\alpha\}$ of $\R[x]_r$  then the bound $\Lasupboundl{r}{K, \mu}$ is equal to the smallest \emph{generalized eigenvalue} of the system:
\begin{equation}
	\label{EQ_generalized_eigenvalue}
	M_{r, f} v = \lambda B_r v,
\end{equation}
where $B_r := M_{r, 1}$ is the matrix with entries $B_r(\alpha, \beta) = \int_K p_\alpha p_\beta d\mu(x)$. Note that if the $p_\alpha$ are orthonormal, then $B_r$ is the identity matrix and one recovers the eigenvalue formulation of Section \ref{SEC_Introduction}. For details, see, e.g., \cite{Lasserre2010}.

This formulation in terms of generalized eigenvalues allows us to work with the standard monomial basis of $\R[x]_r$. To compute the entries of the matrices $M_{r, f}$ and $B_r$, we therefore only require knowledge of the moments:
\[
	\int_\mainset x^\alpha d \mu(x) \quad (\alpha \in \N^n_r).
\]
For the hypercube, simplex and unit ball, closed form expressions for these moments are known (see, e.g., Table 1 in \cite{deKlerkLaurentSurvey}). For the octagon, they can then be computed by triangulation.
We solve the generalized eigenvalue problem \eqref{EQ_generalized_eigenvalue} using the \texttt{eig} function of the SciPy software package.

\medskip
\noindent\textbf{The linear case.}
We consider first the linear case $f(x) = f_{li}(x) = x_1$ and $K = [-1, 1]^2$ equipped with the Lebesgue measure.
Figure \ref{FIG_plot_toy_example} shows the values of the parameters $\ubError{r}{\mainset}{f_{li}}$ and $\ubError{r}{\mainset}{f_{li}} \cdot r^2$.
In accordance with Theorem \ref{THM_general_box} (and \ref{THM_Cheby_box}.(ii)), it appears indeed that $\ubError{r}{\mainset}{f_{li}} = \bigO(1/r^2)$, as suggested by the fact that the parameter $\ubError{r}{\mainset}{f_{li}} \cdot r^2$ approaches a constant value as $r$ grows.

\begin{figure}
	\centering
	\begin{tikzpicture}
\begin{axis}[
	width= 0.47 * \textwidth,
    title={},
    xlabel={r},
    grid=both,
    ylabel={$\ubError{r}{[-1, 1]^2}{f_{li}}$},
    xmin=1, xmax=20,
    ymin=0, ymax=0.5,
    xtick={1,5,10,15,20},
    ytick={0, 0.25, 0.5},
    legend pos=north west,
    cycle list name = exotic,
   	]
\addplot table{data/toy_example.dat};
\end{axis}
\end{tikzpicture}
	\begin{tikzpicture}
\begin{axis}[
	width= 0.47 * \textwidth,
    title={},
    xlabel={r},
    grid=both,
    ylabel={$\ubError{r}{[-1, 1]^2}{f_{li}} \cdot r^2$},
    xmin=1, xmax=20,
    ymin=0, ymax=3,
    xtick={1,5,10,15,20},
    legend pos=north west,
    cycle list name = exotic,
   	]
\addplot table{data/toy_example_r2.dat};
\end{axis}
\end{tikzpicture}
	\caption{The error of upper bounds for $f(x) = x_1 $ computed on $[-1, 1]^2$ w.r.t. the Lebesgue measure.}
	\label{FIG_plot_toy_example}
\end{figure}
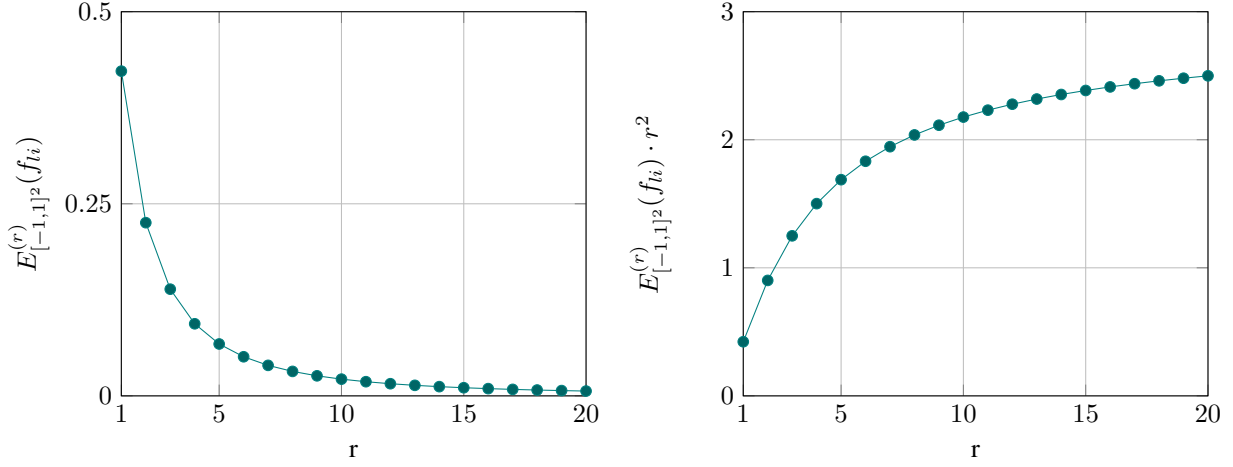

\medskip
\noindent\textbf{The unit ball.}
Next, we consider the unit ball $B^2$, again equipped with the Lebesgue measure. Figure \ref{FIG_plot_box_and_ball} shows the values of the ratio
\begin{equation}
	\label{EQ_box_ball_fraction}
	\ubError{r}{B^2}{f_{*}} / \ubError{r}{[-1,1]^2}{f_{*}}
\end{equation}
for $* \in \{li, qu, bo, ma, ca, mo \}$. In each case, the ratio \eqref{EQ_box_ball_fraction} appears to tend to a constant value, suggesting that the error $\ubError{r}{K}{f_{*}}$ has similar asymptotic behaviour for $K = [-1,1]^2$ and $K = B^2$. This matches the result of Theorem \ref{THM_general_ball} both in the case of a minimizer on the boundary ($* \in \{li, qu \}$) and in the case of a minimizer in the interior ($* \in \{bo, ma, ca, mo \}$).

\begin{figure}[]
	\begin{tikzpicture}
\begin{axis}[
	width=\textwidth,
	height=\textwidth * 0.5,
    title={},
    xlabel={r},
    grid=both,
    ylabel={$\ubError{r}{B^2}{f_{*}} \big / \ubError{r}{[-1,1]^2}{f_{*}}$},
    xmin=1, xmax=20,
    ymin=0, ymax=2,
    xtick={1,5,10,15,20},
    legend pos=north west,
    legend style={at={(0,1)},anchor=north west},
    legend columns= 3,
    cycle list name = exotic,
   	]
\addplot table{data/2d_linear.dat};
\addplot table{data/2d_quadratic.dat};
\addplot table{data/2d_booth.dat};
\addplot table{data/2d_matyas.dat};
\addplot table{data/2d_camel.dat};
\addplot table{data/2d_motzkin.dat};

\legend{Linear, Quadratic, Booth, Matyas, Camel, Motzkin}

\end{axis}
\end{tikzpicture}
	\caption{Comparison of the errors of upper bounds for the functions in Table \ref{TAB_test_functions} computed on $[-1, 1]^2$ and the unit ball $B^2$ w.r.t. the Lebesgue measure}
	\label{FIG_plot_box_and_ball}
\end{figure}
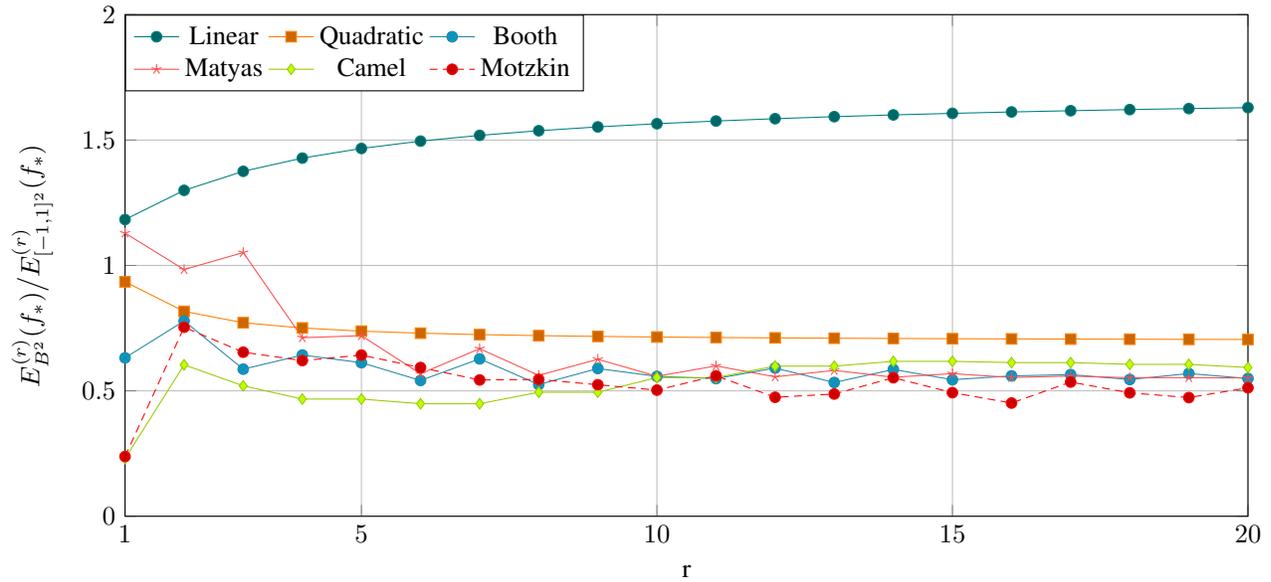
\begin{figure}[]
	\begin{tikzpicture}
\begin{axis}[
	width=\textwidth,
	height=\textwidth * 0.5,
    title={},
    xlabel={r},
    grid=both,
    ylabel={${\ubError{r}{\octagon}{f_{*}}} \big / {\ubError{r}{[-1,1]^2}{f_{*}}}$},
    xmin=1, xmax=20,
    ymin=0, ymax=2.5,
    xtick={1,5,10,15,20},
    legend pos=north west,
    legend style={at={(0,1)},anchor=north west},
    legend columns= 3,
    cycle list name = exotic,
   	]
\addplot table{data/2d_linear_octagon.dat};
\addplot table{data/2d_quadratic_octagon.dat};
\addplot table{data/2d_booth_octagon.dat};
\addplot table{data/2d_matyas_octagon.dat};
\addplot table{data/2d_camel_octagon.dat};
\addplot table{data/2d_motzkin_octagon.dat};

\legend{Linear, Quadratic, Booth, Matyas, Camel, Motzkin}

\end{axis}
\end{tikzpicture}
	\caption{Comparison of the errors of upper bounds for the functions in Table \ref{TAB_test_functions} computed on $[-1, 1]^2$ and the regular octagon $\octagon$ (see \eqref{EQ:octagon}) w.r.t. the Lebesgue measure}
	\label{FIG_plot_box_and_octagon}
\end{figure}
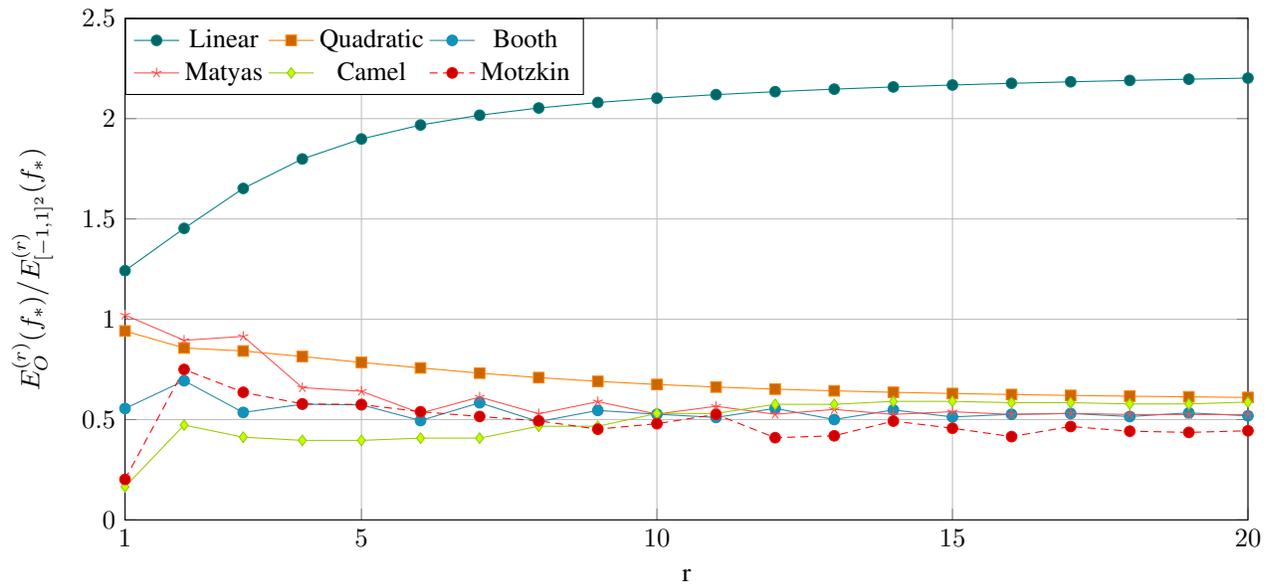
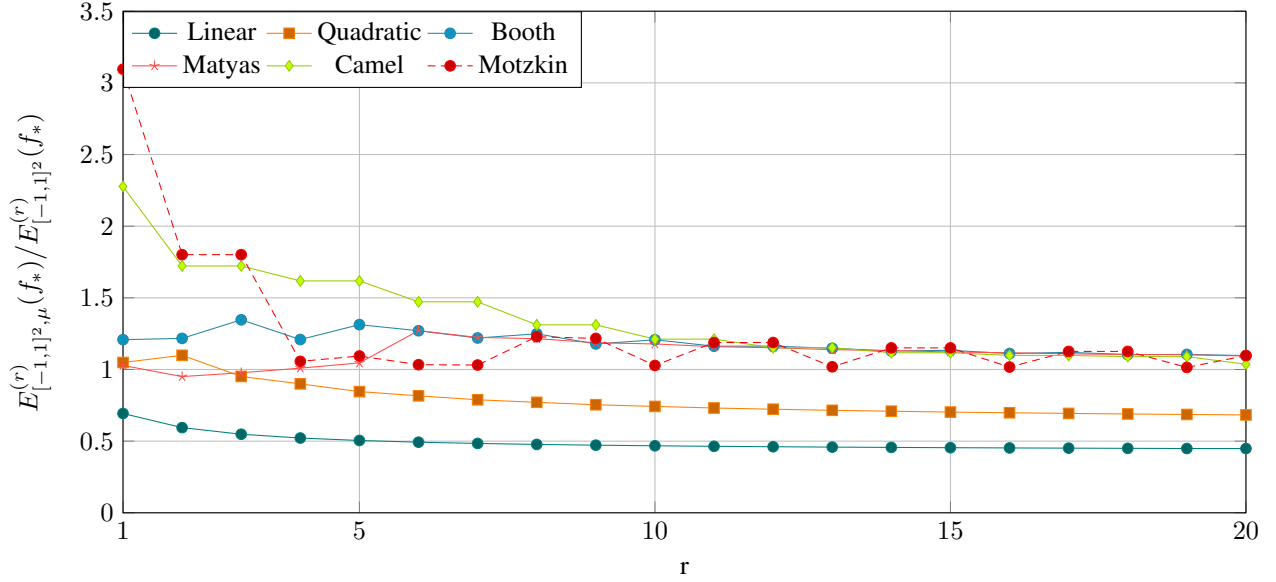
\begin{figure}[]
	\begin{tikzpicture}
\begin{axis}[
	width=\textwidth,
	height=\textwidth * 0.5,
    title={},
    xlabel={r},
    grid=both,
    ylabel={${\ubError{r}{[-1, 1]^2, \mu}{f_{*}}} \big / {\ubError{r}{[-1,1]^2}{f_{*}}}$},
    xmin=1, xmax=20,
    ymin=0, ymax=3.5,
    xtick={1,5,10,15,20},
    legend pos=north west,
    legend style={at={(0,1)},anchor=north west},
    legend columns= 3,
    cycle list name = exotic,
   	]
\addplot table{data/2d_linear_chebyshev.dat};
\addplot table{data/2d_quadratic_chebyshev.dat};
\addplot table{data/2d_booth_chebyshev.dat};
\addplot table{data/2d_matyas_chebyshev.dat};
\addplot table{data/2d_camel_chebyshev.dat};
\addplot table{data/2d_motzkin_chebyshev.dat};

\legend{Linear, Quadratic, Booth, Matyas, Camel, Motzkin}

\end{axis}
\end{tikzpicture}
	\caption{Comparison of the errors of upper bounds for the functions in Table \ref{TAB_test_functions} computed on $[-1, 1]^2$ w.r.t. the Lebesgue and Chebyshev measures.}
	\label{FIG_plot_box_chebyshev}
\end{figure}

\medskip
\noindent \textbf{The regular octagon.}
Consider now the regular octagon (with the Lebesgue measure)
\begin{equation}
	\label{EQ:octagon}
	\octagon = \conv \{ (\pm 1, 0), (0, \pm 1),  (\pm \frac{1}{2}\sqrt{2}, \pm \frac{1}{2}\sqrt{2}) \} \subseteq [-1,1]^2,
\end{equation}
which is an example of a convex body that is not ball-like (see Definition \ref{DEF:round}). Note that as a result, the strongest theoretical guarantee we have shown for the convergence rate of the Lasserre bounds on $\octagon$ 
 is in $\bigO( \log^2 r / r^2)$ (see Theorem \ref{THM_bound_convex_body}).
  Figure \ref{FIG_plot_box_and_octagon} shows the values of the ratio
\begin{equation}
	\label{EQ_box_octagon_fraction}
	\ubError{r}{\octagon}{f_{*}} / \ubError{r}{[-1,1]^2}{f_{*}}
\end{equation}
for $* \in \{li, qu, bo, ma, ca, mo \}$. As for  the unit ball, the ratio \eqref{EQ_box_octagon_fraction} seemingly tends to a constant value for each of the test polynomials. This  indicates a similar asymptotic behaviour of the error $\ubError{r}{K}{f_{*}}$ for $K = [-1, 1]^2$ and $K = \octagon$ and suggests that the convergence rate guaranteed by Theorem \ref{THM_bound_convex_body} might not be tight in this instance.

\medskip
\noindent \textbf{The Chebyshev measure}
Finally, we consider the Chebyshev measure $d\mu(x) = (1-x_1^2)^{-1/2}(1-x_2^2)^{-1/2}dx$ on $[-1, 1]^2$, which we compare to the Lebesgue measure.  Figure \ref{FIG_plot_box_chebyshev} shows the values of the fraction
\begin{equation}
	\label{EQ_box_chebyshev_fraction}
	\ubError{r}{[-1,1]^2, \mu}{f_{*}} / \ubError{r}{[-1,1]^2}{f_{*}}
\end{equation}
for $* \in \{li, qu, bo, ma, ca, mo \}$. Again, we observe that the fraction \eqref{EQ_box_chebyshev_fraction} appears to tend to a constant value in each case, matching the result of Theorem \ref{THM_general_box}.
\section{Concluding remarks}
    \label{SEC_Conclusion}
    \noindent
{\bf Extension to non-polynomial functions.}
Throughout, we have assumed that the function $\mainfunc$ is a polynomial. Strictly speaking, this assumption is not necessary to obtain our results.  For the results in Section~\ref{SEC_SpecialConvexBodies} and in Theorem~\ref{THM_bound_convex_body}, it suffices that $\mainfunc$ has an upper estimator, exact at one of its global minimizers on $\mainset$, and satisfying the properties given in Lemma~\ref{LEM_quadraticupperestimator}. In light of Taylor's Theorem, such an upper estimator exists for all $\mainfunc \in C^2(\R^n, \R)$. For Theorem~\ref{THM_bound_interior_cone}, it is even sufficient that $\mainfunc$ satisfies $\mainfunc(x) \leq \mainfunc(a) + \Lipconst{\mainfunc} ||x - a||$ for all $x \in \mainfunc$, where $\Lipconst{\mainfunc} > 0$ is a constant. That is, it suffices that $\mainfunc$ is Lipschitz continuous on $\mainset$. 
Finally, as shown in \cite[Theorem 10]{deKlerkLaurent2019}, results on the convergence rate of the bounds $f^{(r)}$ for polynomials $f$ extend directly to the case of rational functions $f$.

\medskip\noindent
{\bf Accelerated convergence results.}
For the minimization of linear polynomials  the convergence rate of the bounds $f^{(r)}$ is shown to be in the order 
$\Theta(1/r^2)$ for the hypercube \cite{deKlerkLaurent2018} and  the unit sphere \cite{deKlerkLaurent2019}.
Hence, for arbitrary polynomials, a quadratic rate is the best we can hope for. On the other hand, if we restrict to a class of functions with additional properties, then a better convergence rate can be shown.
Indeed,  a faster convergence rate can be achieved when the function $f$ has many vanishing derivatives at a global minimizer.
We will make use of the following consequence of Taylor's theorem.
\begin{thm}[Taylor's theorem]
Assume $f\in C^\beta(\R^n,\R)$ with $\beta\ge 1$. 
Then we have
$$ f(x)\le \sum_{\alpha\in\N^n, |\alpha|\le \beta-1} {1\over \alpha!}(D^\alpha)(f)(a) (x-a)^\alpha + \delta_{K,f}\|x-a\|^\beta \quad \text{ for all } x\in K$$
for some constant $\delta_{K,f}>0$.
\end{thm}

\begin{thm}
\label{THM_fasterconvergence}
Let $f\in C^\beta(\R^n,\R)$ ($\beta \ge 1$) and let $a$ be a global minimizer of $f$ on $K$. Assume that all partial derivatives $(D^\alpha f)(a)$ vanish for $1\le |\alpha|\le \beta-1$.
Then, given any $\epsilon>0$ we have
\[ 
	\ubError{r}{}{\mainfunc} = \bigO\bigg(\frac{\log^\beta r}{r^\beta}\bigg) = o\bigg(\frac{1}{r^{\beta - \epsilon}}\bigg).
\]
\end{thm}
\begin{proof}
This follows as a direct application of Proposition~\ref{PROP_bound_interior_cone_univariate}.
\qedMP\end{proof}
This applies, e.g., for the univariate polynomial $f(x)=x^{\beta}$ on the interval $K=[0,1]$. 

\noindent
As an application we can answer in the negative a question posed in 
\cite{deKlerkLaurent2018}, where the authors asked about the existence of a `saturation result' for the convergence rate of the Lasserre upper bounds, namely whether 
\begin{equation}
    \ubError{r}{}{\mainfunc} = o\bigg(\frac{1}{r^2}\bigg) \overset{?}{\iff} f \text{ is a constant polynomial}.
\end{equation}

\medskip\noindent
{\bf Application to the generalized problem of moments (GPM) and cubature rules.}
As shown in \cite{deKlerkKuhnPostek2018} results on the convergence analysis of the bounds $\ubError{r}{}{\mainfunc}$ have direct implications for the following generalized moment problem (GMP):
$$\text{\rm val}:= \inf \big\{\int_K f_0(x)d\nu(x): \int_Kf_i(x)d\nu(x)=b_i\ (i\in [m])\big\},$$
where $b_i\in \R$ and  $f_i\in \R[x]$ are given,  and the variable $\nu$ is a Borel measure on $K$. Bounds can be obtained by searching for measures of the form $q_rd\mu$ with  $\mu$ a given Borel measure on $K$ and $q_r\in \Sigma_r$. Their quality can be analyzed via the parameter
$$
	\Delta(r)=\min_{q_r\in \Sigma_r}\max_{i=0}^m \big|\int_K f_i(x)q_r(x)d\mu(x)-b_i\big|,
$$
setting $b_0=\text{\rm val}$. It is shown in \cite{deKlerkKuhnPostek2018} (see also \cite{deKlerkLaurentSurvey}) that, if $\ubError{r}{K,\mu}{\mainfunc} =\bigO(\epsilon(r))$ for all polynomials $f$, then 
$\Delta (r)=\bigO(\sqrt {\epsilon(r)})$. Hence, our results in this paper imply directly that $\Delta (r)=\bigO(\log r /r)$ for general convex bodies and $\bigO(1/r)$ for hypercubes, balls and simplices (recall Table~\ref{TAB_convergence_rates} for exact details). An important instance of (GMP) is finding cubature schemes for numerical integration on $K$ (see, e.g., \cite{deKlerkLaurentSurvey} and references therein).  If $\{x^{(j)},\lambda_j: j\in [N]\}$ form a cubature scheme with positive weights $\lambda_j>0$ that permits to integrate any polynomial of degree at most $d+2r$ on $K$ w.r.t. measure $\mu$, then, as shown in \cite{Martinez2019}, we have
$$f^{(r)}_{K,\mu} \ge f_{cub}^{(r)}:= \min_{j=1}^N f(x^{(j)}) \ge f_{\min,K}.$$
Hence any upper bound on $f^{(r)}_{K,\mu}$ directly gives an upper bound on the parameter $f^{(r)}_{cub}$. Conversely, any lower bound on $f^{(r)}_{cub}$ implies a lower bound on $f^{(r)}_{K,\mu}$, which is the fact used  in \cite{deKlerkLaurent2018, deKlerkLaurent2019} to show the lower bound $\Omega(1/r^2)$ for the hypercube and the sphere.

Finally, let us mention that the needle polynomials are used already in \cite{Kroo2015} to study cubature rules. There, the author considers degree $r$ cubature rules for which the sum $\sum_{j \in [N]}|\lambda_j|$ is polynomially bounded in $r$. For all $x \in \mainset, A \subseteq \mainset$, define the parameters
\begin{equation}
	\rho(x, A) := \sup \{ h > 0 : \ball{h}{x} \cap A = \emptyset \}, ~ \rho(K, A) := \sup_{x \in \mainset} \rho(x, A),
\end{equation}
which indicate how densely $A$ is distributed at $x$ or in $\mainset$, respectively. Kro\'o \cite{Kroo2015} shows that if $X_r$ is the set of nodes of a degree $r$ cubature rule on a convex body $\mainset$, we then have
\begin{equation}
	\rho(\mainset, X_r) = \bigO(\log r / r)
\end{equation}
and that, if $x_0 \in \mainset$ is a vertex of $\mainset$, we even have
\begin{equation}
	\rho(x_0, X_r) = \bigO(\log^2 r / r^2).
\end{equation}
Although the asymptotic rates here are the same as the ones we find in Theorem~\ref{THM_bound_interior_cone} and Theorem~\ref{THM_bound_convex_body}, we are not aware of any direct link between the density of cubature points and the convergence rate of the Lasserre upper bounds.

\medskip\noindent
{\bf Some open questions.}
There are several natural questions left open by this work. The first natural question is whether the convergence rate in $\bigO(1/r^2)$ can be proved for all convex bodies. So far we can only prove a rate in $\bigO(\log^2r/r^2)$, but we suspect that the $\log r$ term is just a consequence of the analysis technique used here. The computational results for the octagon in Section \ref{SEC_NumericalExperiments} seem to support this. Another question is whether this also applies to general compact sets under Assumption~\ref{ASSU_one}, since we know of no example showing this is not possible.

In particular, it is interesting to determine the exact rate of convergence for  polytopes. We could so far only deal with hypercubes and simplices. The main tool we used was the `local similarity' of the simplex with the hypercube. For a general polytope $K$, if the minimum is attained at a point lying in the interior of $K$ or of one of its facets, then we can still apply the `local similarity' tool (and deduce the $\bigO(1/r^2)$ rate). However, at other points (like its vertices) $K$ is in general not locally similar to the hypercube, so another  proof technique seems needed.
A possible strategy could be splitting $K$ into simplices and using the known convergence rate for the simplex containing a global minimizer; however,  a difficulty there is keeping track of the distribution of mass of an optimal sum-of-squares on the other simplices.

    \subsection*{Acknowledgments}
    This work is supported by the Europeans Union’s EU Framework Programme for Research and Innovation Horizon 2020 under the Marie Sk{ł}odowska-Curie Actions Grant Agreement No 764759 (MINOA).
\bibliographystyle{Stylefiles/spmpsci}
\bibliography{main_bib}
\end{document}